\documentclass[oneside,english]{amsart}
\usepackage{lmodern}
\usepackage[T1]{fontenc}
\usepackage[latin9]{inputenc}
\usepackage{verbatim}
\usepackage{amsthm}
\usepackage{amsbsy}
\usepackage{amssymb}
\usepackage{xargs}[2008/03/08]
\PassOptionsToPackage{normalem}{ulem}
\usepackage{ulem}

\makeatletter
\numberwithin{equation}{section}
\numberwithin{figure}{section}
\theoremstyle{plain}
\newtheorem{thm}{\protect\theoremname}[section]
  \theoremstyle{definition}
  \newtheorem{defn}[thm]{\protect\definitionname}
  \theoremstyle{remark}
  \newtheorem{acknowledgement}[thm]{\protect\acknowledgementname}
  \theoremstyle{plain}
  \newtheorem{lem}[thm]{\protect\lemmaname}
  \theoremstyle{plain}
  \newtheorem{cor}[thm]{\protect\corollaryname}
  \theoremstyle{remark}
  \newtheorem{rem}[thm]{\protect\remarkname}
  \theoremstyle{plain}
  \newtheorem{prop}[thm]{\protect\propositionname}
  \theoremstyle{plain}
  \newtheorem{fact}[thm]{\protect\factname}
  \theoremstyle{definition}
  \newtheorem{example}[thm]{\protect\examplename}
  \theoremstyle{remark}
  \newtheorem{claim}[thm]{\protect\claimname}
  \theoremstyle{remark}
  \newtheorem{notation}[thm]{\protect\notationname}
  \theoremstyle{definition}
  \newtheorem{problem}[thm]{\protect\problemname}

\usepackage{amsmath}

\usepackage{a4wide}
\usepackage{url}
\newcommand{\nice}{good}

\makeatother

\usepackage{babel}
  \providecommand{\acknowledgementname}{Acknowledgement}
  \providecommand{\claimname}{Claim}
  \providecommand{\corollaryname}{Corollary}
  \providecommand{\definitionname}{Definition}
  \providecommand{\examplename}{Example}
  \providecommand{\factname}{Fact}
  \providecommand{\lemmaname}{Lemma}
  \providecommand{\notationname}{Notation}
  \providecommand{\problemname}{Problem}
  \providecommand{\propositionname}{Proposition}
  \providecommand{\remarkname}{Remark}
\providecommand{\theoremname}{Theorem}

\begin{document}
\global\long\def\p{\mathbf{p}}
\global\long\def\q{\mathbf{q}}
\global\long\def\C{\mathfrak{C}}
\global\long\def\SS{\mathcal{P}}
 \global\long\def\concat{\frown}
\global\long\def\cl{\operatorname{cl}}
\global\long\def\tp{\operatorname{tp}}
\global\long\def\id{\operatorname{id}}
\global\long\def\cons{\left(\star\right)}
\global\long\def\qf{\operatorname{qf}}
\global\long\def\ai{\operatorname{ai}}
\global\long\def\dtp{\operatorname{dtp}}
\global\long\def\acl{\operatorname{acl}}
\global\long\def\nb{\operatorname{nb}}
\global\long\def\limb{{\lim}}
\global\long\def\leftexp#1#2{{\vphantom{#2}}^{#1}{#2}}
\global\long\def\intr{\operatorname{interval}}
\global\long\def\atom{\emph{at}}
\global\long\def\I{\mathfrak{I}}
\global\long\def\uf{\operatorname{uf}}
\global\long\def\ded{\operatorname{ded}}
\global\long\def\Ded{\operatorname{Ded}}
\global\long\def\Df{\operatorname{Df}}
\global\long\def\Th{\operatorname{Th}}
\global\long\def\eq{\operatorname{eq}}
\global\long\def\Aut{\operatorname{Aut}}
\global\long\def\ac{ac}
\global\long\def\DfOne{\operatorname{df}_{\operatorname{iso}}}
\global\long\def\modp#1{\pmod#1}
\global\long\def\sequence#1#2{\left\langle #1\,\middle|\,#2\right\rangle }
\global\long\def\set#1#2{\left\{  #1\,\middle|\,#2\right\}  }
\global\long\def\Diag{\operatorname{Diag}}
\global\long\def\Nn{\mathbb{N}}
\global\long\def\mathrela#1{\mathrel{#1}}
\global\long\def\twiddle{\mathord{\sim}}
\global\long\def\mathordi#1{\mathord{#1}}
\global\long\def\Qq{\mathbb{Q}}
\global\long\def\dense{\operatorname{dense}}
 \global\long\def\cof{\operatorname{cof}}
\global\long\def\otp{\operatorname{otp}}
\global\long\def\treeexp#1#2{#1^{\left\langle #2\right\rangle _{\tr}}}
\global\long\def\x{\times}
\global\long\def\forces{\Vdash}
\global\long\def\Vv{\mathbb{V}}
\global\long\def\Uu{\mathbb{U}}
\global\long\def\tauname{\dot{\tau}}
\global\long\def\ScottPsi{\Psi}
\global\long\def\cont{2^{\aleph_{0}}}
\newcommandx\MA[2][usedefault, addprefix=\global, 1=]{{MA#1}_{#2}}
\global\long\def\rank#1#2{R_{#1}\left(#2\right)}

\def\Ind#1#2{#1\setbox0=\hbox{$#1x$}\kern\wd0\hbox to 0pt{\hss$#1\mid$\hss} \lower.9\ht0\hbox to 0pt{\hss$#1\smile$\hss}\kern\wd0} 
\def\Notind#1#2{#1\setbox0=\hbox{$#1x$}\kern\wd0\hbox to 0pt{\mathchardef \nn="3236\hss$#1\nn$\kern1.4\wd0\hss}\hbox to 0pt{\hss$#1\mid$\hss}\lower.9\ht0 \hbox to 0pt{\hss$#1\smile$\hss}\kern\wd0} 
\def\nind{\mathop{\mathpalette\Notind{}}}

\global\long\def\ind{\mathop{\mathpalette\Ind{}}}
 \global\long\def\nind{\mathop{\mathpalette\Notind{}}}
\global\long\def\weakLin#1{{\bf K}_{#1}}
 \global\long\def\fullLin#1{{\bf K}_{#1}^{+}}
\global\long\def\interLin#1{{\bf K}_{#1}^{*}}
\global\long\def\collaps#1#2{coll\left(#1,#2\right)}
\global\long\def\eqT#1#2{\sim_{T}^{#1,#2}}
\global\long\def\cal#1{\mathcal{#1}}

\title{Forcing a countable structure to belong to the ground model}

\author{Itay Kaplan and Saharon Shelah}

\thanks{This research was partially supported by the ISRAEL SCIENCE FOUNDATION
(grant No. 1533/14).}

\thanks{The research leading to these results has received funding from the
European Research Council, ERC Grant Agreement n. 338821. No. 1054
on the second author's list of publications. }

\address{Itay Kaplan \\
The Hebrew University of Jerusalem\\
Einstein Institute of Mathematics \\
Edmond J. Safra Campus, Givat Ram\\
Jerusalem 91904, Israel}

\email{kaplan@math.huji.ac.il}

\urladdr{https://sites.google.com/site/itay80/ }

\address{Saharon Shelah\\
The Hebrew University of Jerusalem\\
Einstein Institute of Mathematics \\
Edmond J. Safra Campus, Givat Ram\\
Jerusalem 91904, Israel}

\address{Saharon Shelah \\
Department of Mathematics\\
Hill Center-Busch Campus\\
Rutgers, The State University of New Jersey\\
110 Frelinghuysen Road\\
Piscataway, NJ 08854-8019 USA}

\email{shelah@math.huji.ac.il}

\urladdr{http://shelah.logic.at/}

\begin{abstract}
Suppose that $P$ is a forcing notion, $L$ is a language (in $\Vv$),
$\dot{\tau}$ a $P$-name such that $P\forces$ ``$\dot{\tau}$ is
a countable $L$-structure''. In the product $P\times P$, there
are names $\dot{\tau_{1}},\dot{\tau_{2}}$ such that for any generic
filter $G=G_{1}\x G_{2}$ over $P\x P$, $\dot{\tau}_{1}\left[G\right]=\dot{\tau}\left[G_{1}\right]$
and $\dot{\tau}_{2}\left[G\right]=\dot{\tau}\left[G_{2}\right]$.
Zapletal asked whether or not $P\x P\forces\dot{\tau}_{1}\cong\dot{\tau}_{2}$
implies that there is some $M\in\Vv$ such that $P\forces\dot{\tau}\cong\check{M}$.
We answer this question negatively and discuss related issues%
\thanks{After this paper was published online, it came to our attention, thanks
to Antonio Montalban, that another paper released recently \cite{Knight2014}
has some similar results. See Remark \ref{rem:AfterRelease} for more. %
}. 
\end{abstract}

\maketitle

\section{\label{sec:The main question}The isomorphism property}

Let us start with describing the motivating question (asked by Zapletal)
for this paper. 

Let $P_{1},P_{2}$ be forcing notions, and let $L$ be a language
(vocabulary) such that both $P_{1}$ and $P_{2}$ force $L$ to be
countable. Suppose that $\tauname_{1}$, $\tauname_{2}$ are, respectively,
$P_{1}$ and $P_{2}$ names for countable $L$ structures whose universe,
we may assume, is $\omega$. We also fix our universe $\Vv$. 

Let $\tauname'_{1}$ be a name in the forcing notion $P_{1}\times P_{2}$
such that for any generic filter $G=G_{1}\x G_{2}$ for $P_{1}\x P_{2}$,
$\tauname_{1}'\left[G\right]=\tauname_{1}\left[G_{1}\right]$, and
similarly define $\tauname_{2}'$. 
\begin{itemize}
\item [($\star$)]Suppose that $P_{1}\x P_{2}\forces\tauname'_{1}\cong\tauname_{2}'$.
Does it follow that for some $L$-structure $M\in\Vv$, $P_{1}\forces\tauname_{1}\cong\check{M}$? 
\end{itemize}
First we make a few remarks.

$\ast$ Note that even if $\tauname_{1}$ is forced to be a finite
structure, it is not immediate that the answer is ``yes''. Here's
an example where we can force a new structure with finite universe.
Let $L=\set{P_{i}}{i<\omega}$ where $P_{i}$ are unary predicates.
Let $P$ be the Cohen forcing adding one new real $\varepsilon\in\Vv^{P}$.
Then in $\Vv^{P}$, we can define the structure $N$ whose universe
is $\left\{ 0\right\} $, and such that $P_{i}^{N}=\emptyset$ iff
$\varepsilon\left(i\right)=0$. However, it turns out that in this
case the answer is ``yes'' in our situation by Remark \ref{rem:why inifinte}
below, so for simplicity we will focus on infinite structures. 

$\ast$ Note also that if the answer to ($\star$) is yes, then if
$P_{2}\forces\left|\check{M}\right|=\aleph_{0}$, then also $P_{2}\forces\tauname_{2}\cong\check{M}$.
Indeed, suppose that $G_{2}$ is a $P_{2}$-generic filter over $\Vv$,
and that $G_{1}$ is $P_{1}$-generic over $\Vv\left[G_{2}\right]$.
Then $G_{1}\x G_{2}$ is $P_{1}\x P_{2}$-generic (see \cite[Lemma 15.9]{JechSetTheory})
over $\Vv$, so $\Vv\left[G_{1}\x G_{2}\right]\models\tauname_{2}\cong\check{M}$.
But in $\Vv\left[G_{2}\right]$, the set of pairs $\left(x,y\right)$
of elements of $\omega^{\omega}$ which code isomorphic $L$-structure
is analytic (i.e., $\boldsymbol{\Sigma}_{1}^{1}$), and analytic properties
are absolute between transitive models of ZFC by Mostowski's absoluteness
(see \cite[Theorem 25.4]{JechSetTheory}), so this is true in $\Vv\left[G_{2}\right]$
as well. Another way to see this is using Scott sentences, see Remark
\ref{rem:Another reason for remark in intro}. 

$\ast$ Finally let us note that we can reduce the question to the
case where $P_{1}=P_{2}$. Consider $P_{1}\x P_{2}\x P_{1}$, and,
abusing notation, let $\tauname_{1}',\tauname_{2}'$ be $P_{1}\x P_{2}\x P_{1}$-names
as above, and let $\tauname_{1}''$ be a $P_{1}\x P_{2}\x P_{1}$-name
such that for any generic filter $G=G_{1}\x G_{2}\x G_{3}$ for $P_{1}\x P_{2}\x P_{1}$,
$\tauname_{1}''\left[G\right]=\tauname_{1}\left[G_{3}\right]$. Then
$P_{1}\x P_{2}\x P_{1}\forces\tauname_{1}'\cong\tauname_{1}''$ (because
for any such generic filter $G$, $G_{1}\x G_{2}$ is $P_{1}\x P_{2}$-generic
and $G_{2}\x G_{3}$ is $P_{2}\x P_{1}$-generic). So for any generic
filter $G=G_{1}\x G_{2}\x G_{3}$ for $P_{1}\x P_{2}\x P_{1}$, in
$\Vv\left[G\right]$, $\tauname_{1}\left[G_{1}\right]\cong\tauname_{1}\left[G_{3}\right]$,
and by Mostowski's absoluteness (see above), the same is true in $\Vv\left[G_{1}\x G_{3}\right]$. 

In order to simplify the discussion, let us introduce the following
definition.
\begin{defn}
\label{def:Isomorphism property}Suppose $P$ is a forcing notion,
$L$ a language such that $P\forces\left|\check{L}\right|\leq\aleph_{0}$,
and $\tauname$ is a $P$-name for an infinite $L$-structure with
universe $\omega$ which satisfies: 
\begin{itemize}
\item $P\x P\forces\tauname_{1}\cong\tauname_{2}$ where $\tauname_{1}$
is a $P\x P$-names such that whenever $G=G_{1}\x G_{2}$ is generic
for $P\x P$, $\tauname_{1}\left[G\right]=\tauname\left[G\right]$
and $\tauname_{2}$ is defined similarly. 
\end{itemize}
Then we say that $\left(P,L,\tauname\right)$ has the \emph{isomorphism
property}. 
\end{defn}
In light of the remarks above, we will focus on the following question.
\begin{itemize}
\item [($\star '$)]Suppose that $\left(P,L,\tauname\right)$ has the isomorphism
property. Does it follow that for some $L$-structure $M\in\Vv$,
$P\forces\tauname\cong\check{M}$?
\end{itemize}
Let us say that a forcing notion $P$ is\emph{ \nice} if the answer
to ($\star'$) is ``yes'' for every $L$ and $\tauname$. In Corollary
\ref{cor:negative answer} we give a more natural forcing theoretic
description of such forcing notions: $P$ is {\nice} iff $P$ does
not collapse $\aleph_{2}$ to $\aleph_{0}$.

Section \ref{sec:Preliminaries} consists of three sub-sections, which
mainly serve to recall classical results and to motivate the proceeding
sections. In Section \ref{sub:product forcing} we give the necessary
facts about product of forcing notions. Section \ref{sub:Scott-sentence}
discusses Scott sentences. Section \ref{sub:scott->theory} translates
the Scott sentence of a structure $M$ to a first order theory $T$
such that $M$ is an atomic model of $T$. 

Section \ref{sec:Translating-the-question} translates ($\star'$)
to a question on the existence of atomic models. Section \ref{sec:existence of atomic models}
finally answers ($\star'$) negatively in full generality, but positively
in some cases (e.g., when $P$ does not collapse $\aleph_{2}$ to
$\aleph_{0}$). In addition we investigate when atomic models exist
under classification theory assumptions. In particular, we prove that
if $T$ is a superstable theory and $A$ a subset of a model of $T$
such that the isolated types over $A$ are dense, then: if $\MA{\kappa}$
holds, and $\left|A\right|\leq\kappa^{+}$, then there is an atomic
model over $A$ (this is Theorem \ref{thm:atomic model MA superstable}).
Without Martin's Axiom, this is not true (Example \ref{exa:superstable-continuum plus}). 

Section \ref{sec:On-linear-orders} deals with linear orders. Namely,
we analyze the situation when $\tauname$ is forced to be a linear
order. We do not reach a definite conclusion but we get some equivalent
formulation of the problem. 
\begin{acknowledgement}
We would like to thank the anonymous referee for his very thorough
report and for his useful suggestions. 
\end{acknowledgement}

\section{\label{sec:Preliminaries}Preliminaries}

\subsection{\label{sub:product forcing}Some remarks on product forcing}

The following lemma is easy, and probably well known.
\begin{lem}
\label{lem:everything in V}Suppose $\Uu_{1}$ and $\Uu_{2}$ are
both transitive models of ZFC such that $\Uu_{1}\subseteq\Uu_{2}$.
Suppose that $P\in\Uu_{1}$ is a forcing notion, $\tauname\in\Uu_{1}$
is a $P$-name, and $x\in\Uu_{2}$. Then, if $p\forces\tauname=\check{x}$
(in $\Uu_{2}$) for some $p\in P$, then $x\in\Uu_{1}$. \end{lem}
\begin{proof}
We prove the lemma for any $P$, $p$, $\tauname$ and $x$ by induction
on $\beta$ --- the rank of $x$ ($\beta$ is the the smallest ordinal
$\alpha$ such that $x\in\Uu_{2,\alpha+1}$). For $\beta=0$ it is
obvious. For $\beta>0$, since $p\forces\tauname=\check{x}$, then
for every $y\in x$, for some (actually for densely many) $q\in P$
stronger than $p$, there is some name $\tauname'\in\Uu_{1}$ ($\tauname'$
is a member of some pair --- condition; name --- in $\tauname$),
such that $q\forces\tauname'=\check{y}$. By induction, $y\in\Uu_{1}$.
So $x$ is the set of $y$'s in $\Uu_{1}$ such that $p\forces\check{y}\in\tauname$,
and hence by specification, $x\in\Uu_{1}$. \end{proof}
\begin{cor}
\label{cor:product equal}Suppose $P_{1},P_{2}$ are forcing notions
and that $\tauname_{1},\tauname_{2}$ are $P_{1}$ and $P_{2}$-names
respectively. As usual we let $\tauname_{1}'$ and $\tauname_{2}'$
be the corresponding names in the product. Then if $\left(p,q\right)\forces\tauname_{1}'=\tauname_{2}'$
for some $\left(p,q\right)\in P_{1}\x P_{2}$, then for some $x\in\Vv$,
$p\forces\tauname_{1}=\check{x}$ and $q\forces\tauname_{2}=\check{x}$. \end{cor}
\begin{proof}
Let $G_{1}$ be a generic filter for $P_{1}$ over $\Vv$ containing
$p$. Let $\Uu_{1}=\Vv$, $\Uu_{2}=\Vv\left[G_{1}\right]$, $P=P_{2}$
, $\tauname=\tauname_{2}$ and $x=\tauname_{1}\left[G_{1}\right]$.
So over $\Uu_{2}$, $q\forces\tauname=\check{x}$, and hence by Lemma
\ref{lem:everything in V}, $x\in\Vv$. So $q\forces\tauname_{2}=\check{x}$.
Similarly, for some $x'\in\Vv$, $p\forces\tauname_{1}=x'$. Finally,
it must be that $x=x'$.
\end{proof}

\begin{cor}
Assume as above that $P_{1},P_{2}$ are forcing notions and let $G_{1}\x G_{2}$
be a $P_{1}\x P_{2}$-generic filter. Then $\Vv\left[G_{1}\right]\cap\Vv\left[G_{2}\right]=\Vv$.\end{cor}
\begin{proof}
Suppose that $x$ is in the intersection. Let $\tauname_{1}$ and
$\tauname_{2}$ be $P_{1}$ and $P_{2}$-names such that $\tauname_{1}\left[G_{1}\right]=\tauname_{2}\left[G_{2}\right]$,
and let $\tauname_{1}'$ and $\tauname_{2}'$ be the corresponding
$P_{1}\x P_{2}$-names. Then for some $\left(p,q\right)\in G_{1}\x G_{2}$,
$\left(p,q\right)\forces\tauname_{1}'=\tauname'_{2}$. By Corollary
\ref{cor:product equal}, for some $x\in\Vv$, $p\forces\tauname_{1}=\check{x}$
and $q\forces\tauname_{2}=\check{x}$, so we are done. 
\end{proof}

\subsection{\label{sub:Scott-sentence}Scott sentence}

Recall that for a countable structure $M$ for a countable language,
the \emph{Scott sentence} of $M$ is an $L_{\omega_{1},\omega}$-sentence
$\ScottPsi$ such that whenever $N\models\ScottPsi$ and $\left|N\right|=\aleph_{0}$,
$N\cong M$. 

We need a precise set theoretic definition and coding of $L_{\lambda,\omega}$-formulas
in order to continue. As usually done, we can code formulas and terms
as objects in $\Vv$. In the following paragraph, we do not distinguish
between a formula (or term) and its code.

For instance, a code for a term is a variable $x\in Var$ (where $Var$
is an infinite large enough set of variables) or a tuple $\left(f,t_{1},\ldots,t_{n}\right)$
where $f$ is an $n$-place function from $L$, and $t_{1},\ldots,t_{n}$
are terms. Similarly we define codes for atomic formulas as tuples
$\left(R,t_{1},\ldots,t_{n}\right)$ where $R$ is an $n$-place relation
symbol and $t_{1},\ldots,t_{n}$ are terms. 

We also fix a constant set for the logical symbols $\neg,\bigwedge,\exists$.
The code for the negation of an $L_{\lambda,\omega}$-formula $\neg\varphi$
is the pair $\left(\neg,\varphi\right)$, and similarly the code for
$\exists x\varphi$ is $\left(\exists,x,\varphi\right)$. The code
for $\bigwedge_{i\in I}\psi_{i}$ (where $\psi_{i}$ are formulas,
and $\left|I\right|<\lambda$) is the pair $\left(\bigwedge,\set{\psi_{i}}{i\in I}\right)$.
The connectors $\bigvee$ and $\to$, and the quantifier $\forall$
are treated as abbreviations. 

As usual, $L_{\infty,\omega}$ is the union of $L_{\lambda,\omega}$
running over all $\lambda$. 
\begin{rem}
\label{rem:L_infinity,omega is absoluite}The property of being a
(code for) $L_{\lambda,\omega}$-sentence for some $\lambda$ is absolute.
I.e., if $\Uu_{1}\subseteq\Uu_{2}$ are transitive models of ZFC having
the same ordinals, $x\in\Uu_{1}$, then $\Uu_{1}\models$``$x$ is
an $L_{\infty,\omega}$-formula (sentence)'' iff the same is true
in $\Uu_{2}$. This can be proved by induction on the rank of $x$
in $\Uu_{1}$. 
\end{rem}

The choice for the coding of $\bigwedge_{i\in I}\psi_{i}$ to be a
set rather than a sequence is important as we shall see now: we give
a canonical construction for Scott sentences as in e.g., \cite{KeislerInf}.
\begin{prop}
\label{prop:Scott}Suppose $L$ is a countable language. There is
a (class) function $Sc$ whose domain is the class of all countable
$L$-structures, and whose range is the set of all (codes for) $L_{\omega_{1},\omega}$
sentences such that for all countable $L$-structures $M$, $N$,
$M\cong N$ iff $N\models Sc\left(M\right)$ iff $Sc\left(M\right)=Sc\left(N\right)$. \end{prop}
\begin{proof}
We repeat the construction from \cite{KeislerInf}. Given a countable
$L$-structure $M$, we define $Sc\left(M\right)$. By induction on
$\alpha<\omega_{1}$, for every finite tuple $\bar{a}=\left(a_{0},\ldots,a_{n-1}\right)\in M^{<\omega}$,
we define an $L_{\omega_{1},\omega}$-formula $\phi_{\alpha,\bar{a}}\left(\bar{x}\right)$
with $\bar{x}=\left(x_{0},\ldots,x_{n-1}\right)$ as follows: 

For $\alpha=0$, $\phi_{\alpha,\bar{a}}\left(\bar{x}\right)=\bigwedge\set{\varphi\left(\bar{x}\right)}{M\models\varphi\left(\bar{a}\right),\,\varphi\mbox{ atomic or a negation of an atomic formula}}$.

For $\alpha$ limit, $\phi_{\alpha,\bar{a}}\left(\bar{x}\right)=\bigwedge\set{\phi_{\beta,\bar{a}}\left(\bar{x}\right)}{\beta<\alpha}$.

For $\alpha=\beta+1$, 
\begin{eqnarray*}
\phi_{\alpha,\bar{a}}\left(\bar{x}\right) & = & \phi_{\beta,\bar{a}}\left(\bar{x}\right)\land\\
 &  & \forall x_{n+1}\bigvee\set{\phi_{\beta,\bar{a}\concat b}\left(\bar{x},x_{n+1}\right)}{b\in M}\land\bigwedge\set{\exists x_{n+1}\phi_{\beta,\bar{a}\concat b}\left(\bar{x},x_{n+1}\right)}{b\in M}.
\end{eqnarray*}
Now, one can prove by induction on $\alpha<\omega_{1}$ that $\bar{a}\models\phi_{\alpha,\bar{a}}\left(\bar{x}\right)$.
For some $\beta<\omega_{1}$, $M\models\forall\bar{x}\left(\phi_{\beta,\bar{a}}\left(\bar{x}\right)\to\phi_{\beta+1,\bar{a}}\left(\bar{x}\right)\right)$
for all $\bar{a}\in M^{<\omega}$. Let $Sc\left(M\right)$ be 
\[
\phi_{\beta,\emptyset}\wedge\bigwedge\set{\forall\bar{x}\left(\phi_{\beta,\bar{a}}\left(\bar{x}\right)\to\phi_{\beta+1,\bar{a}}\left(\bar{x}\right)\right)}{\bar{a}\in M^{<\omega}}.
\]

By a back and forth argument as in \cite{KeislerInf}, if $N\models Sc\left(M\right)$
then $M\cong N$. In fact, if $\bar{a}\in M^{<\omega}$ and $\bar{b}\in N^{<\omega}$,
and $M\models\phi_{\beta,\bar{c}}\left(\bar{a}\right)$, $N\models\phi_{\beta,\bar{c}}\left(\bar{b}\right)$
for some $\bar{c}\in M^{<\omega}$ of the same length as $\bar{a}$
and $\bar{b}$, then there is an isomorphism between $M$ and $N$
taking $\bar{a}$ to $\bar{b}$.

For the other direction, suppose $f:M\to N$ is an isomorphism. It
is easy to see by induction on $\alpha<\omega_{1}$ that for every
$\bar{a}\in M^{<\omega}$, $\phi_{\alpha,\bar{a}}=\phi_{\alpha,f\left(\bar{a}\right)}$
(in the induction step we rely on the choice of coding --- as sets
and not sequences). In particular, $Sc\left(M\right)=Sc\left(N\right)$. \end{proof}
\begin{prop}
\label{prop:Scott sentence in V}Assume that $\left(P,L,\tauname\right)$
has the isomorphism property. Then there is an $L_{\infty,\omega}$-sentence
$\ScottPsi$ in $\Vv$ such that $P\forces$``$\check{\ScottPsi}$
is the Scott sentence of $\tauname$''. Furthermore, there is a complete
first order theory $T\in\Vv$ in the language $L$ such that $P\forces\check{T}=Th\left(\tauname\right)$.\end{prop}
\begin{proof}
Let $\dot{\ScottPsi}$ be a name for $Sc\left(\tauname\right)$ in
$P$. Then $P\forces$``$\dot{\ScottPsi}$ is an $L_{\omega_{1},\omega}$-sentence''.
Let $\dot{\ScottPsi}_{1}$ and $\dot{\ScottPsi}_{2}$ be $P\x P$-names
such that for any generic $G=G_{1}\x G_{2}$ for $P\x P$, $\dot{\ScottPsi}_{1}\left[G\right]=\dot{\ScottPsi}\left[G_{1}\right]$
and similarly for $\dot{\ScottPsi}_{2}$. Then $P\x P\forces\dot{\ScottPsi}_{1}=\dot{\ScottPsi}_{2}$
by Proposition \ref{prop:Scott}. Hence by Corollary \ref{cor:product equal},
for some $\ScottPsi\in\Vv$, $P\forces\dot{\ScottPsi}=\check{\ScottPsi}$,
so by Remark \ref{rem:L_infinity,omega is absoluite}, $\ScottPsi$
is an $L_{\infty,\omega}$-sentence and we are done. 

The furthermore part, regarding the first order theory, is proved
similarly. \end{proof}
\begin{rem}
\label{rem:Another reason for remark in intro}In the slightly different
context of Section \ref{sec:The main question}, where we had two
forcing notions $P_{1}$ and $P_{2}$ and two names $\tauname_{1}$
and $\tauname_{2}$, if $P_{1}\x P_{2}\forces\tauname'_{1}\cong\tauname_{2}'$
and for some $L$-structure $M\in\Vv$, $P_{1}\forces\tauname_{1}\cong\check{M}$,
then Proposition \ref{prop:Scott sentence in V} gives another reason
why in this case, if $P_{2}\forces\left|\check{M}\right|=\aleph_{0}$,
then also $P_{2}\forces\tauname_{2}\cong\check{M}$ (without using
Mostowski's absoluteness). Why? as in the proof of said proposition,
we can show that the Scott sentence $\ScottPsi$ of $\tauname_{1}$
(which is the same as the one of $\tauname_{2}$) is in $\Vv$. But
then $M\models\ScottPsi$ (in $\Vv$) so $P_{2}\forces\check{M}\models\ScottPsi$
and we are done by Proposition \ref{prop:Scott}. 
\end{rem}

\begin{rem}
\label{rem:why inifinte}Why did we restrict to the case where the
structures are infinite? if $\left(P,L,\tauname\right)$ has the isomorphism
property but $P\forces\left|\tauname\right|<\omega$, then the first
order theory $T$ in $L$ which $P$ forces to be $Th\left(\tauname\right)$
is in $\Vv$ by Proposition \ref{prop:Scott sentence in V}. Hence
$T$ is a consistent first order theory in $L$, and hence has a model
$M\in\Vv$. But now $P\forces\check{M}\equiv\tauname$, and as $M$
is finite, they must be isomorphic. 
\end{rem}

\subsection{\label{sub:scott->theory}from a Scott sentence to a theory}

The outline for this subsection is as follows. First we show the well
known result that an $L_{\lambda^{+},\omega}$-sentence $\ScottPsi$
is ``the same thing as'' a first order theory $T_{\ScottPsi}$ and
a collection of types to be omitted. We then show that when $\ScottPsi$
is a Scott sentence of a structure $M$, then in the language of $T_{\ScottPsi}$,
$M$ is an atomic model of $T_{\ScottPsi}^{co}$ --- the natural completion
of $T_{\ScottPsi}$. In the end of this section we remark that $T_{\ScottPsi}^{co}$
can be constructed directly from $T_{\ScottPsi}$ in an absolute way,
without passing through $M$ (hence, when $M$ is constructed in some
generic extension but $\ScottPsi$ is already in the ground model,
then so is $T_{\ScottPsi}^{co}$). 
\begin{lem}
\label{lem:languae,theory,types}Suppose that $\ScottPsi$ is a consistent
$L_{\lambda^{+},\omega}$-sentence in the language $L$. Then there
is a language $L_{\ScottPsi}$, a consistent first order $L_{\ScottPsi}$-theory
$T_{\ScottPsi}$ of size $\leq\lambda$, a collection $\Gamma_{\ScottPsi}$
of partial $T_{\ScottPsi}$-types and a canonical bijection $H_{\ScottPsi}$
which respects isomorphisms between the class of $L$-structures $M\models\ScottPsi$
and the class of $L_{\ScottPsi}$ structures $N\models T_{\ScottPsi}$
such that $N$ omits all the types in $\Gamma_{\ScottPsi}$.\end{lem}
\begin{proof}
Recall that for an $L_{\infty,\omega}$ formula $\psi$, a sub-formula
is an $L_{\infty,\omega}$ formula that appears in the construction
of $\psi$. So for instance, sub-formulas of $\psi=\bigwedge\set{\varphi_{i}}{i\in I}$
are $\psi$ and sub-formulas of $\varphi_{i}$ for all $i$, but \uline{not}
any conjunction $\bigwedge\set{\varphi_{i}}{i\in I'}$ for $I'\subseteq I$. 

We may assume that $\left|L\right|\leq\lambda$, by restricting to
symbols which appear in $\ScottPsi$. We may also assume that $\ScottPsi$
contains as sub-formulas all formulas of the form $P\left(x_{0},\ldots,x_{n-1}\right)$
and $F\left(x_{0},\ldots,x_{n-1}\right)=x_{n}$ for every $n$-place
relation symbol $P$ and $n$-place function symbol $F$. Otherwise,
we replace $\ScottPsi$ with $\ScottPsi\wedge\varphi$ where $\varphi$
is a big conjunction of sentences of the form $\forall\bar{x}\left(P\left(\bar{x}\right)\vee\neg P\left(\bar{x}\right)\right)$
and $\forall\bar{x}\exists yF\left(\bar{x}\right)=y$.

Let $L_{\ScottPsi}$ be comprised of $n$-ary relation symbols $R_{\varphi}$
for every sub-formula $\varphi\left(x_{0},\ldots,x_{n-1}\right)$
of $\ScottPsi$ (note that any sub-formula has a finite number of
free variables as $\ScottPsi$ is a sentence). So $\left|L_{\ScottPsi}\right|\leq\lambda$.
The theory $T_{\ScottPsi}$ will have the axioms:
\begin{itemize}
\item $\exists x_{n}R_{\varphi}\left(x_{0},\ldots,x_{n}\right)\leftrightarrow R_{\exists x_{n}\varphi}\left(x_{0},\ldots,x_{n-1}\right)$
whenever $\exists x_{n}\varphi\left(x_{0},\ldots,x_{n}\right)$ is
a sub-formula of $\ScottPsi$. 
\item $\neg R_{\varphi}\leftrightarrow R_{\neg\varphi}$ whenever $\neg\varphi$
is a sub-formula of $\ScottPsi$. 
\item $R_{\bigwedge\set{\varphi_{i}}{i\in I}}\to R_{\varphi_{i}}$ for every
$i\in I$ whenever $\bigwedge\set{\varphi_{i}}{i\in I}$ is a sub-formula
of $\ScottPsi$. 
\item $R_{\ScottPsi}$. 
\end{itemize}
The set of types $\Gamma_{\ScottPsi}$ consists of types of the form
$\Sigma_{\Phi}=\set{R_{\varphi_{i}}\left(\bar{x}\right)}{i\in I}\cup\left\{ \neg R_{\bigwedge\set{\varphi_{i}}{i\in I}}\left(\bar{x}\right)\right\} $
whenever $\Phi=\bigwedge\set{\varphi_{i}\left(\bar{x}\right)}{i\in I}$
is a sub-formula of $\ScottPsi$ (and $\bar{x}$ a tuple of variables).

Finally, given an $L$-structure $M$, the induced $L_{\ScottPsi}$-structure
$M_{*}=H_{\ScottPsi}\left(M\right)$ has the same universe, and for
every $R_{\varphi}\in L_{\ScottPsi}$, $R_{\varphi}^{M_{*}}=\varphi^{M}$.
Note that $M_{*}$ omits all the types in $\Gamma_{\ScottPsi}$, and
since $\ScottPsi$ is consistent, $T_{\ScottPsi}$ is consistent.
Also, if $M_{1}\cong M_{2}$ then $H_{\ScottPsi}\left(M_{1}\right)\cong H_{\ScottPsi}\left(M_{2}\right)$. 

Conversely, given a model $N$ of $T_{\ScottPsi}$ which omits all
types in $\Gamma_{\ScottPsi}$, define an $L$-structure $N_{*}$,
which in fact will be $H_{\ScottPsi}^{-1}\left(N\right)$, by $P^{N_{*}}=\left(R_{P\left(x_{0},\ldots,x_{n-1}\right)}\right)^{N}$
for every $n$-place predicate $P\in L$, and similarly for every
$n$-place function symbol $F\in L$. Note that this map also respects
isomorphisms. 
\end{proof}

\begin{rem}
\label{rem:absoluteness of G_psi}The map $\ScottPsi\mapsto\left(L_{\ScottPsi},T_{\ScottPsi},\Gamma_{\ScottPsi},H_{\ScottPsi}\right)$
is absolute in the sense that if $\ScottPsi\in\Vv$ and $\Uu$ is
a transitive model of ZFC extending $\Vv$, then $\left(L_{\ScottPsi},T_{\ScottPsi},\Gamma_{\ScottPsi}\right)$
are the same in $\Uu$ as in $\Vv$ and $H_{\ScottPsi}$ defines the
same function when restricted to $\Vv$. This is easily seen by induction
on the rank of $\ScottPsi$. 
\end{rem}
Recall:
\begin{defn}
Let $T$ be a consistent first order theory (not necessarily complete),
and let $\Sigma\left(x_{0},\ldots,x_{n-1}\right)$ be a partial type.
We say that a formula $\varphi\left(x_{0},\ldots,x_{n-1}\right)$
\emph{isolates} $\Sigma$ in $T$ if $T\cup\left\{ \exists\bar{x}\varphi\left(\bar{x}\right)\right\} $
is consistent and $T\vdash\varphi\to\psi$ for every $\psi\in\Sigma$.
The partial type $\Sigma$ is \emph{isolated} in $T$ if some formula
isolates it in $T$.
\end{defn}

\begin{defn}
A structure $M$ for a language $L$ is called \emph{atomic} if for
every finite tuple $\bar{a}\in M^{<\omega}$, $\tp\left(\bar{a}/\emptyset\right)$
is isolated in $Th\left(M\right)$. Similarly, we say that a set $A\subseteq M$
is \emph{atomic} if for every finite tuple $\bar{a}\in A^{<\omega}$,
$\tp\left(\bar{a}/\emptyset\right)$ is isolated in $Th\left(M\right)$.
We add ``over $B$'' for some set $B\subseteq M$ to mean that we
replace $\emptyset$ with $B$. \end{defn}
\begin{prop}
\label{prop:Gpsi is atomic}Suppose that $M$ is a countable structure,
and that $\ScottPsi=Sc\left(M\right)$ is its Scott sentence (see
Proposition \ref{prop:Scott}). Then the induced theory $T_{\ScottPsi}$
is countable and has a natural completion, $T_{\ScottPsi}^{co}=Th\left(H_{\ScottPsi}\left(M\right)\right)$.
Furthermore, the induced $L_{\ScottPsi}$-structure $H_{\ScottPsi}\left(M\right)$
is atomic. \end{prop}
\begin{proof}
The theory $T_{\ScottPsi}$ is countable since $\ScottPsi$ is in
$L_{\omega_{1},\omega}$. This also implies that $\Gamma_{\ScottPsi}$
is countable. 

In order to show that $H_{\ScottPsi}\left(M\right)$ is atomic, take
any tuple $\bar{a}\in M^{<\omega}$. We will show that $R_{\phi_{\beta,\bar{a}}}$
isolates $\tp\left(\bar{a}/\emptyset\right)$ (recall the construction
of $\ScottPsi=Sc\left(M\right)$ in Proposition \ref{prop:Scott}).
Suppose that $\theta\left(\bar{x}\right)\in\tp\left(\bar{a}/\emptyset\right)$,
and $H_{\ScottPsi}\left(M\right)\models\exists\bar{x}\left(R_{\phi_{\beta,\bar{a}}}\land\neg\theta\left(\bar{x}\right)\right)$.
Suppose that $\bar{b}\models R_{\phi_{\beta,\bar{a}}}\left(\bar{x}\right)\land\neg\theta\left(\bar{x}\right)$
where $\bar{b}$ is from $M$. But then $M\models\phi_{\beta,\bar{a}}\left(\bar{b}\right)$,
so there is an automorphism $\sigma$ of $M$ which takes $\bar{b}$
to $\bar{a}$ (see the proof of Proposition \ref{prop:Scott}). But
then $\sigma$ is also an automorphism of $H_{\ScottPsi}\left(M\right)$
--- a contradiction.\end{proof}
\begin{rem}
\label{rem:abs of completion}In the same context of Proposition \ref{prop:Gpsi is atomic},
the completion $T_{\ScottPsi}^{co}$ can also be constructed directly
from $T_{\ScottPsi}$ without passing through $M$. Namely, for an
ordinal $\alpha$ define a consistent $L_{\ScottPsi}$-theory $T_{\ScottPsi}^{\alpha}$
by induction on $\alpha$ by: $T_{\ScottPsi}^{0}=T_{\ScottPsi}$;
for $\alpha$ limit, take $T_{\ScottPsi}^{\alpha}=\bigcup_{\beta<\alpha}T_{\ScottPsi}^{\beta}$;
finally, for $\alpha=\beta+1$, define 

\[
T_{\ScottPsi}^{\alpha}=T_{\ScottPsi}^{\beta}\cup\set{\neg\exists\bar{x}\left(\varphi\left(\bar{x}\right)\right)}{\varphi\left(\bar{x}\right)\mbox{ isolates }\Sigma_{\Phi}\mbox{ in }T_{\ScottPsi}^{\beta}\mbox{ for some sub-formula }\Phi\mbox{ of }\ScottPsi}.
\]

(Recall $\Sigma_{\Phi}$ from the construction of $T_{\ScottPsi}$in
Lemma \ref{lem:languae,theory,types} above.)

Since $L_{\ScottPsi}$ is countable we must get stuck at some countable
ordinal $\alpha$, and we let $T_{\ScottPsi}^{co}=T_{\ScottPsi}^{\alpha}$.
Since $H_{\ScottPsi}\left(M\right)$ omits all the types in $\Gamma_{\ScottPsi}$,
$H_{\ScottPsi}\left(M\right)\models T_{\ScottPsi}^{\beta}$ for all
$\beta$, and in particular $H_{\ScottPsi}\left(M\right)\models T_{\ScottPsi}^{\alpha}$.
The types in $\Gamma_{\ScottPsi}$ are not isolated in $T_{\ScottPsi}^{\alpha}$
(if $\varphi\left(\bar{x}\right)$ isolates some $\Sigma_{\Phi}$
in $T_{\ScottPsi}^{\alpha}$, then $T_{\ScottPsi}^{\alpha+1}=T_{\ScottPsi}^{\alpha}$
includes $\neg\exists\bar{x}\left(\varphi\left(\bar{x}\right)\right)$).
Moreover, this is true for any consistent extension $T_{\ScottPsi}^{\alpha}\cup\left\{ \theta\right\} $
for a sentence $\theta$ (if $\varphi\left(\bar{x}\right)$ isolates
some type in $T_{\ScottPsi}^{\alpha}\cup\left\{ \theta\right\} $,
then $\varphi\left(\bar{x}\right)\wedge\theta$ isolates it in $T_{\ScottPsi}^{\alpha}$).
By the omitting type theorem for countable languages (see \cite[Theorem 4.2.4]{Marker}),
$T_{\ScottPsi}^{\alpha}\cup\left\{ \theta\right\} $ has a countable
model $N$ which omits all the types in $\Gamma_{\ScottPsi}$, and
so by Lemma \ref{lem:languae,theory,types}, $H_{\ScottPsi}^{-1}\left(N\right)\models\ScottPsi$,
and so $H_{\ScottPsi}^{-1}\left(N\right)\cong M$ and hence $N\cong H_{\ScottPsi}\left(M\right)$.
We conclude that for any such $\theta$, $T_{\ScottPsi}^{\alpha}\vdash\theta$.
Hence $T_{\ScottPsi}^{\alpha}$ is complete. This construction of
$T_{\ScottPsi}^{co}$ is absolute from $\ScottPsi$. 
\end{rem}

\section{\label{sec:Translating-the-question}Translating the question}

Here we will translate ($\star'$) from Section \ref{sec:The main question}
to a question about the existence of atomic models. 

Suppose that $\left(P,L,\tauname\right)$ has the isomorphism property.
By Proposition \ref{prop:Scott sentence in V}, there is an $L_{\infty,\omega}$
sentence $\ScottPsi_{\tauname}$ in $\Vv$ which $P$ forces to be
$\tau$'s Scott sentence. It makes sense then to let $L_{\ScottPsi_{\tauname}}$,
$T_{\ScottPsi_{\tauname}}$, $T_{\ScottPsi_{\tauname}}^{co}$ and
$\Gamma_{\ScottPsi_{\tauname}}$ be the induced language, theories
and collection of types as in Lemma \ref{lem:languae,theory,types}.
By Remarks \ref{rem:absoluteness of G_psi} and \ref{rem:abs of completion},
they are the same in $\Vv$ as in $\Vv\left[G\right]$ for any generic
filter $G$. 

Recall the following definition. 
\begin{defn}
Suppose $T$ is a first order theory. We say that the\emph{ isolated
types are dense in $T$} if whenever $\varphi\left(\bar{x}\right)$
is a consistent formula, i.e., $T\cup\left\{ \exists\bar{x}\varphi\right\} $
is consistent (where $\bar{x}$ is a finite tuple of variables), there
is a consistent formula $\theta\left(\bar{x}\right)$ such that $T\vdash\theta\to\varphi$
and $\theta$ isolates a complete type: for every formula $\psi\left(\bar{x}\right)$,
either $T\models\theta\to\psi$ or $T\models\theta\to\neg\psi$.\end{defn}
\begin{fact}
\label{fac:atomic equivalent}\cite[Thoerem 4.2.10]{Marker} Suppose
that $T$ is a countable complete theory, then the following are equivalent:
\begin{itemize}
\item $T$ has a countable atomic model.
\item The isolated types are dense in $T$. 
\end{itemize}
\end{fact}

\begin{fact}
\label{fac:atomic isomorphic} \cite[Theorem 4.2.8, Corollary 4.2.16]{Marker}
Suppose that $M$ and $N$ are countable atomic models of a countable
complete theory $T$. Then $M\cong N$. \end{fact}
\begin{thm}
\label{thm:main translation}Suppose that $P$ is a forcing notion,
$L$ a language which $P$ forces to be countable, and $\tauname$
a $P$-name for an $L$-structure with universe $\omega$ such that
$\left(P,L,\tauname\right)$ has the isomorphism property. Then the
following are equivalent:
\begin{enumerate}
\item For some $M\in\Vv$, $P\forces\tauname\cong\check{M}$. 
\item $T_{\ScottPsi_{\tauname}}^{co}$ has an atomic model.
\end{enumerate}
\end{thm}
\begin{proof}
(1) $\Rightarrow$ (2): By (1), for some $M\in\Vv$, $P\forces\tauname\cong\check{M}$.
Then $M\models\ScottPsi_{\tauname}$, and $H_{\ScottPsi_{\tauname}}\left(M\right)\models T_{\ScottPsi_{\tauname}}^{co}$.
Let $G$ be a generic filter for $P$. Then by Proposition \ref{prop:Gpsi is atomic},
in $\Vv\left[G\right]$, $H_{\ScottPsi_{\tauname}}\left(M\right)$
is an atomic model of $T_{\ScottPsi_{\tauname}}^{co}$, but being
an atomic model is absolute, hence the same is true in $\Vv$. 

(2) $\Rightarrow$ (1): Let $G$ be any generic filter. Then by Proposition
\ref{prop:Gpsi is atomic}, in $\Vv\left[G\right]$, $T_{\ScottPsi_{\tauname}}^{co}$
is countable and $H_{\ScottPsi_{\tauname}}\left(\tauname\left[G\right]\right)$
is an atomic model of $T_{\ScottPsi_{\tauname}}^{co}$. By (2), $T_{\ScottPsi_{\tauname}}^{co}$
has an atomic model $N$ which we can assume has cardinality $\leq\left|L_{\ScottPsi_{\tauname}}\right|=\aleph_{0}^{\Vv\left[G\right]}$
by taking an elementary substructure. Let $M=H_{\ScottPsi_{\tauname}}^{-1}\left(N\right)$.
In $\Vv\left[G\right]$, $N$ is countable, so isomorphic to $H_{\ScottPsi_{\tauname}}\left(\tauname\left[G\right]\right)$
in $\Vv\left[G\right]$ by Fact \ref{fac:atomic isomorphic}, hence
$M\cong\tauname\left[G\right]$.
\end{proof}
Recall that a forcing notion $P$ is {\nice} when the answer to ($\star'$)
is ``yes'' for every $L$ and $\tauname$: whenever $\left(P,L,\tauname\right)$
has the isomorphism property, for some $L$-structure $M\in\Vv$,
$P\forces\tauname\cong\check{M}$. 
\begin{thm}
\label{thm:main translation2}Suppose that $P$ is a forcing notion.
Then the following are equivalent:
\begin{enumerate}
\item $P$ is \nice. 
\item For every complete first order theory $T$ such that the isolated
types are dense in $T$ and $P\forces\left|\check{T}\right|=\aleph_{0}$,
$T$ has an atomic model.
\end{enumerate}
\end{thm}
\begin{proof}
(1) $\Rightarrow$ (2): Suppose $T$ is a complete first order theory
in a language $L$ in which the isolated types are dense, and $P\forces\left|T\right|=\aleph_{0}$.
Let $\tauname$ be a name for a countable atomic model of $T$ (exists
by Fact \ref{fac:atomic equivalent}). Then by Fact \ref{fac:atomic isomorphic},
$\left(P,L,\tauname\right)$ has the isomorphism property. By (1),
for some $M\in\Vv$, $P\forces\tauname\cong\check{M}$. Let $G$ be
a generic filter for $P$. Then in $\Vv\left[G\right]$, $M$ is an
atomic model of $T$, but being an atomic model of $T$ is absolute,
hence the same is true in $\Vv$. 

(2) $\Rightarrow$ (1): Suppose that $\left(P,L,\tauname\right)$
has the isomorphism property. By Theorem \ref{thm:main translation},
it is enough to show that $T_{\ScottPsi_{\tauname}}^{co}$ has an
atomic model, so by (2) it is enough to show that the isolated types
are dense. Let $G$ be a generic filter. Then by Proposition \ref{prop:Gpsi is atomic},
in $\Vv\left[G\right]$, $T_{\ScottPsi_{\tauname}}^{co}$ is complete
and countable and $H_{\ScottPsi_{\tauname}}\left(\tauname\left[G\right]\right)$
is an atomic model of $T_{\ScottPsi_{\tauname}}^{co}$. Hence by fact
\ref{fac:atomic equivalent}, the isolated types are dense in $T_{\ScottPsi_{\tauname}}^{co}$.
But this is an absolute property, hence the same is true in $\Vv$. 
\end{proof}

\section{\label{sec:existence of atomic models}On the existence of atomic
models}

In Section \ref{sub:A-criterion:-collapsing}, we give a general criterion
for when $P$ is \nice. In Sections \ref{sub:Totally-transcendental-theories}
and \ref{sub:Superstable-theories} we investigate when an atomic
model exist under classification theoretic assumption.

\subsection{\label{sub:A-criterion:-collapsing}A criterion: collapsing $\aleph_{2}$
to $\aleph_{0}$}

Having translated ($\star'$) to a question on the existence of atomic
models in Theorem \ref{thm:main translation2}, we can now start to
provide some answers. In Corollary \ref{cor:not-collaps-ok}, we provide
a positive answer (i.e., $P$ is \nice), provided that $P$ does
not collapse $\aleph_{2}$ to $\aleph_{0}$. Conversely, in Corollary
\ref{cor:negative answer} we prove that if $P$ does collapse $\aleph_{2}$
to $\aleph_{0}$ then $P$ is not \nice. 

First we note that by Fact \ref{fac:atomic equivalent} the following
is immediate. 
\begin{cor}
If $P$ is a forcing notion that does not collapse $\aleph_{1}$ then
$P$ is \nice.
\end{cor}
For completeness we provide a proof of the following proposition,
which is really an adaptation of \cite[Chapter IV, Theorem 5.5]{Sh:c}.
This theorem was also proved independently by Julia Knight \cite[Theorem 1.3]{knight}
and David Kueker \cite[Page 168]{kueker}.
\begin{prop}
\label{prop:Aleph1}Suppose $\left|T\right|=\aleph_{1}$ and the isolated
types are dense in $T$. Then $T$ has an atomic model. \end{prop}
\begin{proof}
Let $M\models T$ be $\aleph_{1}$-saturated. Construct an atomic
set $N=\set{b_{i}}{i<\omega_{1}}\subseteq M$ such that for any formula
$\varphi\left(x,\bar{b}\right)$, with $\bar{b}$ a finite tuple from
$N$, if $M\models\exists x\varphi\left(x,\bar{b}\right)$ then for
some $c\in N$, $M\models\varphi\left(c,\bar{b}\right)$. Then $N$
is an atomic model of $T$. By an easy book-keeping argument, it is
enough to show that if $A\subseteq M$ is a countable atomic set,
$\bar{b}$ is a finite tuple from $A$, and $M\models\exists x\varphi\left(x,\bar{b}\right)$,
then for some $c\in M$, $M\models\varphi\left(c,\bar{b}\right)$
and $A\cup\left\{ c\right\} $ is atomic. 

For a consistent formula $\psi\left(\bar{x}\right)$, choose a consistent
formula $\theta_{\psi}\left(\bar{x}\right)$ which isolates a complete
type and implies $\psi$. 

Enumerate $A=\set{a_{i}}{i<\omega}$, and assume that $\bar{b}=\left(a_{0},\ldots,a_{n-1}\right)$.
For $i<\omega$, let $\bar{a}_{i}=\left(a_{0},\ldots,a_{i-1}\right)$,
and let $\theta_{i}\left(\bar{z}_{i}\right)$ isolate $\tp\left(\bar{a}_{i}\right)$.
Construct a sequence of formulas $\psi_{i}\left(x,\bar{z}_{i}\right)$
such that: for $i\geq n$: $\psi_{n}\left(x,\bar{z}_{n}\right)\to\varphi\left(x,\bar{z}_{n}\right)$;
$\psi_{i}$ is consistent; $T\models\psi_{i+1}\to\psi_{i}$; $\psi_{i}$
isolates a complete type and $\bar{a}_{i}\models\exists x\psi_{i}\left(x,\bar{z}_{i}\right)$.
The construction: $\psi_{n}=\theta_{\varphi\left(x,\bar{z}_{n}\right)\land\theta_{n}\left(\bar{z}_{n}\right)}$
and $\psi_{i+1}=\theta_{\psi_{i\left(x,\bar{z}_{i}\right)}\land\theta_{i+1}\left(\bar{z}_{i+1}\right)}$.
Finally, $\set{\psi_{i}\left(x,\bar{a}_{i}\right)}{n\leq i<\omega}$
is consistent, so let $c\in M$ satisfy this type. \end{proof}
\begin{cor}
\label{cor:not-collaps-ok}If $P$ is a forcing notion that does not
collapse $\aleph_{2}$ to $\aleph_{0}$ then $P$ is \nice. \end{cor}
\begin{fact}
\cite{Laskowski-Shelah}\label{fact:Shelah-Laskowski} There is a
complete first order theory $T$ with a sort $V$ whose elements form
an indiscernible set in any model $M$ of $T$, such that for any
set $A\subseteq V^{M}$, the isolated types in $T\left(A\right)$
are dense but if $\left|A\right|\geq\aleph_{2}$, $T\left(A\right)$
has no atomic model.\end{fact}
\begin{cor}
\label{cor:negative answer}If $P$ collapses $\aleph_{2}$ to $\aleph_{0}$
then $P$ is not \nice. 

In conclusion, $P$ is {\nice} iff $P$ does not collapse $\aleph_{2}$
to $\aleph_{0}$. 
\end{cor}

\subsection{\label{sub:Totally-transcendental-theories}Totally transcendental
theories}
\begin{defn}
Recall that a complete first order theory $T$ is called\emph{ totally
transcendental }if there is no sequence of formulas $\sequence{\varphi_{s}\left(\bar{x},\bar{y}_{s}\right)}{s\in2^{<\omega}}$
such that in some model $M$ of $T$ there are tuples $\sequence{\bar{a}_{s}}{s\in2^{<\omega}}$
such that for each $\eta\in2^{\omega}$ the type $\set{\varphi_{\eta\upharpoonright n}\left(\bar{x},\bar{a}_{\eta\upharpoonright n}\right)^{\eta\left(n\right)}}{n<\omega}$
is consistent (where $\varphi^{0}=\neg\varphi$, $\varphi^{1}=\varphi$).\end{defn}
\begin{fact}
\label{fac:tt has atomic}\cite[Lemma 5.3.4, Corollary 5.3.7]{TentZiegler}
If $T$ is totally transcendental, then $T$ has an atomic model. 
\end{fact}
Note that for a complete first order theory $T$, being totally transcendental
is an absolute property. This follows from the fact that having an
infinite branch in a tree is an absolute property. Namely, define
the tree $\Sigma$ consisting of finite families of sequences of formulas
of the form $\sequence{\varphi_{s}\left(\bar{x},\bar{y}_{s}\right)}{s\in2^{<n}}$
such that it is consistent with $T$ that there are tuples $\sequence{\bar{a}_{s}}{s\in2^{<n}}$
with the property that for each $t\in2^{n}$ the type $\set{\varphi_{t\upharpoonright k}\left(\bar{x},\bar{a}_{t\upharpoonright k}\right)^{t\left(k\right)}}{k<n}$
is consistent. The order between two such finite families is
\[
\sequence{\varphi_{s}\left(\bar{x},\bar{y}_{s}\right)}{s\in2^{<n}}\leq\sequence{\psi_{s}\left(\bar{x},\bar{y}_{s}\right)}{s\in2^{<m}}
\]
 iff $n\leq m$ and for each $s\in2^{<n},$ $\varphi_{s}=\psi_{s}$.
This is easily seen to be a set theoretic tree. Then $T$ is totally
transcendental iff $\Sigma$ has no infinite branch. This is a $\Delta_{1}$
property, and hence absolute, see \cite[Lemma 13.11]{JechSetTheory}. 
\begin{defn}
\label{def:omega stable}Recall that a complete first order theory
$T$ in a countable language is called \emph{$\omega$-stable} if
for every countable model $M\models T$, the number of complete $1$-types
over $M$ is at most $\aleph_{0}$. 
\end{defn}
It is easy to see that if $T$ is $\omega$-stable then it is also
totally transcendental. The converse is also true when $T$ is countable.
See \cite[Theorem 5.2.6]{TentZiegler}. 
\begin{cor}
Suppose that $\left(P,L,\tauname\right)$ has the isomorphism property,
and that $\ScottPsi=\ScottPsi_{\tauname}\in L_{\infty,\omega}$ is
the Scott sentence of $\tauname$. Let $T_{\ScottPsi}^{co}$ be the
induced theory. Then if $P\forces$``$T_{\ScottPsi}^{co}$ is $\omega$-stable'',
then for some $M\in\Vv$, $P\forces\tauname\cong\check{M}$. \end{cor}
\begin{proof}
Since $P$ forces that $T_{\ScottPsi}^{co}$ is $\omega$-stable and
countable, it follows that $P\forces$``$T_{\ScottPsi}^{co}$ is totally
transcendental''. But then $T_{\ScottPsi}^{co}$ is totally transcendental
by absoluteness. So $T_{\ScottPsi}^{co}$ has an atomic model by Fact
\ref{fac:tt has atomic}, and hence by Theorem \ref{thm:main translation}
we are done. 
\end{proof}
Warning: it is tempting to think that if $\left(P,L,\tauname\right)$
has the isomorphism property and $P\forces$``$Th\left(\tauname\right)$
is $\omega$-stable'' then $T_{\ScottPsi_{\tauname}}^{co}$ is $\omega$-stable.
However this is not the case.
\begin{example}
It is well known that there is a graph on $\omega$ with language
$\left\{ S\right\} $ (so $S$ is a 2-place relation) such that the
full theory $\left(\Nn,+,\cdot,0,1\right)$ can be interpreted in
$\left(\omega,S\right)$ see e.g., \cite[Theorem 5.5.1]{Hod}. Let
$L=\left\{ P,Q,\pi_{1},\pi_{2}\right\} \cup\set{Q_{k}}{k<\omega}$
where $P,Q$ are unary predicates, $\pi_{1},\pi_{2}$ are unary function
symbols, and for $k<\omega$, $Q_{k}$ is a unary predicate. Let $M$
be the following $L$-structure: its universe is the union of $P^{M}\cup Q^{M}$
where $P^{M}=\omega$ and 
\[
Q^{M}=\set{\left(n,m,\alpha\right)}{n,m<\omega,\alpha\leq\omega,\left(n\mathrela Sm\to\alpha<\omega\right)}.
\]
The function $\pi_{1}^{M}:Q^{M}\to P^{M}$ is the projection to the
first coordinate ($\pi_{1}^{M}\left(n,m,\alpha\right)=n$) and $\pi_{2}^{M}:Q^{M}\to P^{M}$
is the projection to the second coordinate (meaning that on $P^{M}$,
$\pi_{1}$ and $\pi_{2}$ are the identity). Finally, $Q_{k}^{M}=\set{\left(n,m,k\right)}{n,m<\omega}$.
Let $T=Th\left(M\right)$. Then it is easy to see that $T$ has quantifier
elimination. If $N\equiv M$ is a countable model, then the number
of $1$-types over $N$ is countable: given $c\notin N$ in some elementary
extension, the type of $c$ over $N$ is determined as follows. If
$c\in P$ its type is the unique (non-algebraic) type. Otherwise the
type of $c$ is determined by the unique $k<\omega$ so that $c\in Q_{k}$
(if there is any) and by the type of the pair $\left(\pi_{1}\left(c\right),\pi_{2}\left(c\right)\right)$
over $N$. Hence $T$ is $\omega$-stable. 

Let $\ScottPsi=\ScottPsi_{M}$ be the Scott sentence of $M$. For
any $a\in Q^{M}\backslash\bigcup_{k<\omega}Q_{k}^{M}$, $\bigwedge_{k<\omega}\neg Q_{k}\left(x\right)$
holds. Even though formally $\varphi\left(x\right)=\bigwedge_{k<\omega}\neg Q_{k}\left(x\right)$
is not a sub-formula of $\ScottPsi$ (because $\phi_{0,a}$ contains
the full atomic type of $a$), $\varphi^{M}$ is a definable set in
$H_{\ScottPsi}\left(M\right)$. It follows that in $H_{\ScottPsi}\left(M\right)$,
$\omega$ is definable (by $R_{P}$) and also $S$: $n\mathrela Sm$
iff $R_{P}\left(n\right)$, $R_{P}\left(m\right)$ and $\forall z\in R_{Q}\left(\left(\pi_{1}\left(z\right)=n\land\pi_{2}\left(z\right)=m\right)\to\neg\left(\varphi^{M}\left(z\right)\right)\right)$.
So a model of $T_{\ScottPsi}$ can interpret the full theory of arithmetic.
In particular, $T_{\ScottPsi}$ is not $\omega$-stable. 
\end{example}

\subsection{\label{sub:Superstable-theories}Superstable theories}

Recall the following definition.
\begin{defn}
A complete first order theory $T$ is superstable if there exists
some cardinal $\lambda$ such that for all model $M\models T$, $\left|S_{1}\left(M\right)\right|\leq\left|M\right|+\lambda$. 
\end{defn}
We start by giving an example, similar in spirit to the one in Fact
\ref{fact:Shelah-Laskowski}, but here we replace $\aleph_{2}$ with
$\left(2^{\aleph_{0}}\right)^{+}$, and the theory involved is superstable. 
\begin{example}
\label{exa:superstable-continuum plus}(Thanks to Chris Laskowski)
Let $N\in\omega$. Let $L_{N}=\left\{ U,V,\pi\right\} \cup\set{E_{n}}{n<N+1}$
where $U,V$ are unary predicates, $\pi$ is a unary function symbol
and for $n<N+1$, $E_{n}$ is a binary relation. Let $T_{N}^{\forall}$
be the following theory:
\begin{itemize}
\item $U,V$ are disjoint.
\item $\pi:U\to V$. 
\item For each $n<N$: $E_{n}$ is an equivalence relation on $U$; $E_{n+1}\subseteq E_{n}$;
$E_{n+1}$ has at most two classes in any $E_{n}$ class; $E_{0}$
has just one class.
\end{itemize}
Now, $T_{N}^{\forall}$ is a universal theory with the amalgamation
and disjoint embedding properties. Hence by e.g., \cite[Theorem 7.4.1]{Hod},
it has a model completion, $T_{N}$, which eliminates quantifiers
and is $\omega$-categorical. One can also check that $T_{N+1}\supseteq T_{N}$.
Let $T=\bigcup_{N<\omega}T_{N}$. 
\begin{claim}
$T$ is superstable but not $\omega$-stable. \end{claim}
\begin{proof}
[Proof of claim] If $M$ is a countable model of $T$, then for each
$n<\omega$, $M$ contains representatives for all the $2^{n}$ classes
of $E_{n}$. Hence there are $2^{\omega}$ many types over $M$, so
$T$ is not $\omega$-stable. However, it is superstable by quantifier
elimination, as there are at most $2^{\aleph_{0}}+\lambda$ 1-types
over a model of size $\lambda$. \end{proof}
\begin{claim}
\label{claim:The-isolated-types are dense}The isolated types are
dense in $T$.\end{claim}
\begin{proof}
[Proof of claim] Suppose that $\psi\left(x_{0},\ldots,x_{n-1}\right)$
is a consistent formula in $T$. Suppose $\psi$ is in $L_{N+1}$
for some $N$. Let $\bar{a}=a_{0},\ldots,a_{n-1}$ realize $\psi$
in some model $M\models T$ and assume that $\bar{a}\subseteq U^{M}$.
There are $2^{N}$ many $E_{N}$-classes, and let us partition $\bar{a}$
into these classes, and suppose the largest such class has $l$ elements.
Let $K\geq N$ be such that $2^{K-N}\geq l$, so each $E_{N}$ classes
contains at least $l$ distinct $E_{K}$-classes. Let $\bar{b}=b_{0},\ldots,b_{n-1}$
in $M$ have the same type as $\bar{a}$ in $L_{N+1}$ but $b_{i}$
and $b_{j}$ are not $E_{K}$-equivalent for $i\neq j$ (such a $\bar{b}$
exists since $M\upharpoonright L_{K+1}$ is an existentially closed
model of $T_{K+1}^{\forall}$, and by our choice of $K$). Let $p\left(\bar{x}\right)$
be the quantifier free type of $\bar{b}$ in $L_{K+1}$, so it is
a finite set, and let $\theta\left(\bar{x}\right)=\bigwedge p$. Then
$T\models\theta\to\psi$ and $\theta$ isolates a complete type (since
if $\bar{c}\models\theta$, then the partition of $\bar{c}$ by any
equivalence relation $E_{i}$ is determined by $\theta$). If $\bar{a}$
contains also some elements from $V$, then a simple adjustment of
the above argument will work. 
\end{proof}
Hence $T$ has an atomic model $M_{0}$ by Fact \ref{fac:atomic equivalent}.
Note that, by quantifier elimination, in any model $N\models T$,
$V^{N}$ is an infinite indiscernible set: all sets of elements of
size $n$ have the same type. Let $N\models T$ be of size $\geq\left(2^{\aleph_{0}}\right)^{+}$
and assume that $N$ contains some $A\subseteq V^{N}$ of size $\geq\left(2^{\aleph_{0}}\right)^{+}$.
Let $T\left(A\right)$ be the expansion of $T$ by the full theory
of $N$ in the language $L\left(A\right)$ (where we add constants
for elements of $A$). 
\begin{claim}
$T\left(A\right)$ is superstable of size $\left|A\right|$. The isolated
types are dense in $T\left(A\right)$ but $T\left(A\right)$ has no
atomic model. \end{claim}
\begin{proof}
[Proof of claim]The first sentence is clear. To see that the isolated
types are dense in $T\left(A\right)$, one can either repeat the proof
of Claim \ref{claim:The-isolated-types are dense}, or note that since
$A$ is an indiscernible set, given a consistent formula $\varphi\left(\bar{x},\bar{a}\right)$
($\bar{a}$ is a tuple from $A$), there is some tuple $\bar{b}$
in $V^{M_{0}}$ with the same type as $\bar{a}$. Suppose that $\psi\left(\bar{y}\right)$
isolates the type of $\bar{b}$. There is a formula $\theta\left(\bar{x},\bar{y}\right)$
which isolates a type and implies $\varphi\left(\bar{x},\bar{y}\right)\land\psi\left(\bar{y}\right)$,
so $\theta\left(\bar{x},\bar{b}\right)$ is consistent, implies $\varphi\left(\bar{x},\bar{b}\right)$
and isolates a complete type and hence so is $\theta\left(\bar{x},\bar{a}\right)$.

Suppose $N\models T\left(A\right)$ is atomic. Since $\pi$ is onto
(which follows from our choice of $T$), $\left|U^{N}\right|>2^{\aleph_{0}}$.
Hence there are $a,b\in U^{N}$ such that $a\mathrela{E_{n}}b$ for
all $n<\omega$. But easily, the type $\tp\left(a,b\right)$ is not
isolated by quantifier elimination. 
\end{proof}
\end{example}
This example gives a weaker version of Corollary \ref{cor:negative answer}.
\begin{cor}
\label{cor:negative answer1}Example \ref{exa:superstable-continuum plus}
shows that if $P$ collapses $\left(2^{\aleph_{0}}\right)^{+}$ to
$\aleph_{0}$ then $P$ is not \nice, and (2) in Theorem \ref{thm:main translation2}
can be witnessed with a superstable theory. 
\end{cor}
Now, we will show that under Martin's Axiom, if $T$ is superstable,
then such an example as in \ref{fact:Shelah-Laskowski} cannot exist. 
\begin{notation}
Let $\left(P,<\right)$ be a partial order. We use the standard interpretation
of ``strength'' in $P$ when we think of it as a forcing notion,
i.e., $a$ is stronger than $b$ if $a<b$. \end{notation}
\begin{defn}
[\textbf{Martin's Axiom}]For an infinite cardinal $\kappa$, let
$\MA{\kappa}$ (Martin's Axiom for $\kappa$) be the following statement:
\begin{itemize}
\item If $\left(P,<\right)$ is partially ordered set that satisfies the
countable chain condition and if $D$ is a collection of at most $\kappa$
dense subsets of $P$, then there exists a $D$-generic set $G$ on
$P$ (i.e., $G$ meets every element of $D$, $G$ is downward directed:
if $p,q\in G$ then there is a condition stronger than both $p,q$,
and if $r$ is weaker than $p$, then $r$ is in $G$).
\end{itemize}
\end{defn}
\begin{rem}
\label{rem:Cohen =00003D countable}We will in fact need a weakening
of Martin's Axiom, namely, $\MA[\left(\mbox{Cohen}\right)]{\kappa}$,
which states that whenever $D$ is a collection of at most $\kappa$
dense subsets of the Cohen forcing (adding one real), then there is
a $D$-generic set $G$ on $P$. It is well known that any non-trivial
countable forcing notion is forcing-equivalent to Cohen forcing (i.e.,
they generate the same generic extension). Hence $\MA[\left(\mbox{Cohen}\right)]{\kappa}$
is equivalent to $\MA[\left(\mbox{countable}\right)]{\kappa}$ (the
restriction of $\MA{\kappa}$ to countable forcing notions). Note
that $\MA[\left(\mbox{Cohen}\right)]{\kappa}$ implies that $\kappa<\cont$. 
\end{rem}
To prove the next theorem we will also have to recall:
\begin{defn}
\label{def:rank}\cite[II, Definition 1.1]{Sh:c} Let $T$ be a complete
first order theory with monster model $\C$. Let $p\left(\bar{x}\right)$
be a partial type. We define by induction on $\alpha$ when is $\rank{\infty}p\geq\alpha$
as follows:
\begin{itemize}
\item $\rank{\infty}p\geq0$ iff $p$ is consistent. 
\item When $\delta$ is a limit ordinal, then $R_{\infty}\left(p\right)\geq\delta$
iff $\rank{\infty}p\geq\alpha$ for all $\alpha<\delta$.
\item $\rank{\infty}p\geq\alpha+1$ iff for every finite $q\subseteq p$,
and for every cardinal $\mu$, there are partial types $\set{q_{i}\left(\bar{x}\right)}{i<\mu}$
such that $\rank{\infty}{q_{i}\cup q}\geq\alpha$, and $q_{i}\cup q_{j}$
are explicitly inconsistent for $i\neq j$ (i.e., there is a formula
$\varphi\left(\bar{x},\bar{y}\right)$ such that for some $\bar{a}\in\C$,
$\varphi\left(\bar{x},\bar{a}\right)\in q_{i}$, $\neg\varphi\left(\bar{x},\bar{a}\right)\in q_{j}$). 
\end{itemize}
\end{defn}
We will need some facts about forking in order to continue. Let $T$
be a complete first order theory, and suppose that $\C$ is its monster
model. Given a set $A\subseteq\C$ (whose size is, as usual, smaller
than $\left|\C\right|$) there is a class of formulas with parameters
in $\C$ which are called the \emph{forking formulas} over $A$. The
precise definition can be found in e.g., \cite[Definition 7.1.7]{TentZiegler}.
Given a tuple $\bar{a}$ and sets $A\subseteq B\subseteq\C$, we write
$\bar{a}\ind_{A}B$ when $p=\tp\left(\bar{a}/B\right)$ does not fork
over $A$, meaning that no formula in $p$ forks over $A$. For our
purposes we will need the following facts:
\begin{fact}
\label{fact:stationary}\cite[Section 8.5]{TentZiegler} Let $A\subseteq B\subseteq\C$.
For any formula $\varphi\left(\bar{x}\right)$, it forks over $A$
iff it forks over $\acl\left(A\right)$ iff any equivalent formula
forks over $A$, so forking is a property of definable sets and not
of formulas. The set of formulas $\varphi\left(\bar{x}\right)$ over
$B$ which fork over $A$ form an ideal (a finite disjunction of forking
formulas forks over $A$, if $\psi$ forks over $A$ and $\varphi\vdash\psi$
then $\varphi$ forks over $A$, and $\bar{x}\neq\bar{x}$ forks over
$A$). As a result, if $q\left(\bar{x}\right)$ is a partial type
over $B$ which does not fork over $A$, then there is a complete
type $q\subseteq p$ over $B$ which does not fork over $A$ (this
is a non-forking extension of $q$). 

If $T$ is stable, then any type over $A$ does not fork over $A$.
If moreover $A=\acl^{\eq}\left(A\right)$ (here we assume $\C=\C^{\eq}$),
then any complete type over $A$ has a unique non-forking extension
to $B$. 
\end{fact}
The connection between forking and ranks is given in:
\begin{fact}
\cite[III, Lemma 1.2]{Sh:c}\label{fact:forking if rank decreases}
Suppose that $p\left(\bar{x}\right)$ is a partial type over $A$,
$\varphi\left(\bar{x},\bar{a}\right)$ a formula and $\rank{\infty}p=\rank{\infty}{p\cup\left\{ \varphi\right\} }<\infty$.
Then $\varphi$ does not fork over $A$. 
\end{fact}

\begin{fact}
\cite[II, Theorem 3.14]{Sh:c} A complete theory $T$ is superstable
iff for any formula $\varphi\left(x_{0},\ldots,x_{n-1}\right)$ (perhaps
with parameters), $\rank{\infty}{\varphi}<\infty$. Superstable theories
are stable. 
\end{fact}
Before stating the main theorem, let us recall another fact:
\begin{lem}
\label{lem:alg closure also atomic}If $M$ is a structure, and $A\subseteq M$
is an atomic set, then so is $\acl\left(A\right)$. \end{lem}
\begin{proof}
The proof is an easy exercise in the definitions. \end{proof}
\begin{thm}
\label{thm:atomic model MA superstable}Assume $\MA[\left(\mbox{Cohen}\right)]{\kappa}$.
Suppose $T$ is superstable and countable, and $A\subseteq M\models T$
has size $\leq\kappa^{+}$. If the isolated types are dense in $T\left(A\right)$,
then $T\left(A\right)$ has an atomic model. \end{thm}
\begin{proof}
Let $\C\models T$ be a monster model of $T$ containing $A$ (so
$\C$ is a reduct of the monster model of $T\left(A\right)$ to the
language $L$ of $T$). We may assume that $\C=\C^{\eq}$. As in the
proof of Proposition \ref{prop:Aleph1}, it is enough to show that
if $B\subseteq\C$ is an atomic set over $A$ of size $\leq\kappa$,
$\bar{b}$ is a finite tuple from $B$, $\bar{a}$ a finite tuple
from $A$, and $\C\models\exists x\varphi\left(x,\bar{b},\bar{a}\right)$,
then for some $c\in\C$, $\C\models\varphi\left(c,\bar{b},\bar{a}\right)$
and $B\cup\left\{ c\right\} $ is atomic over $A$. 

Since $T$ is superstable, there is some consistent formula $\psi\left(x,\bar{c}\right)$
over $D=\acl^{\eq}\left(A\cup B\right)$ such that $\psi\vdash\varphi$,
and $\alpha=\rank{\infty}{\psi}$ is minimal among all consistent
formulas over $D$ which imply $\varphi$. Let $C=\acl^{\eq}\left(\bar{c}\right)\subseteq D$,
so $C$ is a countable set (the algebraic closure is taken in $T$
and not in $T\left(A\right)$). Let $\pi\left(x\right)$ be the partial
type over $D$ consisting of all formulas of the form $\neg\chi\left(x,\bar{e}\right)$
where $\bar{e}$ is from $D$ and $\chi$ forks over $C$ (equivalently,
$\chi$ forks over $\bar{c}$). 

We will say that two formulas $\xi\left(x\right),\zeta\left(x\right)$
(with parameters from $\C$) are\emph{ equivalent modulo $\pi$} if
$\pi\vdash\zeta\leftrightarrow\xi$. 
\begin{claim}
\label{claim:equivalent modulo pi}If $\xi\left(x,\bar{d}\right)$
is a consistent formula over $D$ such that $\xi\vdash\psi$ then
$\xi$ is equivalent modulo $\pi$ to a consistent formula over $C$.\end{claim}
\begin{proof}
[Proof of claim]By choice of $\psi$, $\rank{\infty}{\xi}=\rank{\infty}{\psi}$.
By Fact \ref{fact:forking if rank decreases}, $\xi\left(x\right)$
does not fork over $C$. Suppose $\C\models\xi\left(e\right)$ for
some $e\models\pi$. Let $p=\tp\left(e/C\right)$. Then $q=\tp\left(e/D\right)$
is a non-forking extension of $p$, but by Fact \ref{fact:stationary},
$q$ is the unique non-forking extension of $p$, so $\pi\cup p\vdash\xi$,
and by compactness, for some $\zeta_{e}\in p$, $\pi\cup\left\{ \zeta_{e}\right\} \vdash\xi$.
Hence by compactness, $\xi$ is equivalent modulo $\pi$ to a finite
disjunction of such formulas $\zeta_{i}\left(x\right)$ over $C$. 
\end{proof}
Let $P$ be the set of consistent formulas $\xi\left(x\right)$ over
$C$ which imply $\psi$. We define an order $<$ on $P$ by: $\xi<\zeta$
iff $\C\models\xi\to\zeta$. Equivalently, $\pi\vdash\xi\to\zeta$
(as $\xi$ and $\zeta$ are formulas over $C$, so if $a\models\xi$,
and $\pi\vdash\xi\to\zeta$, then $p=\tp\left(a/C\right)$ does not
fork over $C$, so we can find some $a'\models p$ such that $a'\models\pi$,
so $a'\models\zeta$, and hence $\zeta\in p$, so $a\models\zeta$).
Note that $P$ is countable.

Fix some finite tuple $\bar{d}$ from $B$. Let $X_{\bar{d}}$ be
the set of formulas $\xi$ from $P$ such that: $\pi\vdash\xi\left(x\right)\to\beta\left(x,\bar{d}',\bar{a}\right)$
for some formula $\beta$, where $\bar{a}$ is a finite tuple from
$A$ and $\bar{d}'$ is a finite tuple from $D$ containing $\bar{d}$,
such that $\beta\left(x,\bar{y},\bar{a}\right)$ isolates a complete
type over $A$.

We claim that $X_{\bar{d}}$ is dense in $P$. Indeed, let $\zeta\left(x,\bar{e}\right)\in P$
(so $\bar{e}$ is a finite tuple from $C$). Since $\bar{e}$ is algebraic
over $\bar{c}$, by Lemma \ref{lem:alg closure also atomic}, $\tp\left(\bar{e}\bar{d}/A\right)$
is isolated, say by $\theta_{\bar{e}\bar{d}}\left(\bar{y},\bar{z},\bar{a}'\right)$.
Let $\beta'\left(x,\bar{y},\bar{z},\bar{a}'\right)=\zeta\left(x,\bar{y}\right)\wedge\theta_{\bar{e}\bar{d}}\left(\bar{y},\bar{z},\bar{a}'\right)$.
This is a consistent formula, so by assumption, there is some consistent
$\beta\left(x,\bar{y},\bar{z},\bar{a}\right)\vdash\beta'\left(x,\bar{y},\bar{z},\bar{a}'\right)$
which isolates a complete type over $A$, where $\bar{a}$ is a finite
tuple from $A$. Note that $\beta\left(x,\bar{e},\bar{d},\bar{a}\right)$
is consistent and implies $\zeta\left(x,\bar{e}\right)$. By Claim
\ref{claim:equivalent modulo pi}, $\beta\left(x,\bar{e},\bar{d},\bar{a}\right)$
is equivalent modulo $\pi$ to some formula $\xi\left(x,\bar{e}'\right)$
over $C$ from $P$. Then $\pi\vdash\xi\to\zeta$ and $\xi\in X_{\bar{d}}$
so we are done. 

By $\MA[\left(\mbox{Cohen}\right)]{\kappa}$, there is some $\set{X_{\bar{d}}}{\bar{d}\in B^{<\omega}}$
-generic set $G\subseteq P$. Note that $G$ is a type (i.e., consistent).
This is because all the elements of $P$ are consistent and since
$G$ is downward directed. 

Being a partial type over $C$, $G$ does not fork over $C$, so $G\cup\pi$
is consistent. Let $c\models G\cup\pi$. 

First of all, $c\models G$ and any formula in $G$ implies $\psi\left(x,\bar{c}\right)$,
which implies $\varphi\left(x,\bar{b},\bar{a}\right)$, so $\C\models\varphi\left(c,\bar{b},\bar{a}\right)$.
Now, suppose that $\bar{d}\concat\left\langle c\right\rangle $ is
some finite tuple from $B\cup\left\{ c\right\} $. By choice of $G$,
there is some $\xi\in G\cap X_{\bar{d}}$, so for some $\beta\left(x,\bar{d}',\bar{a}\right)$
as above, $\beta\left(c,\bar{d}',\bar{a}\right)$ holds (since $c\models\pi$),
$\bar{d}'=\bar{d}\concat\bar{d}''$ for some $\bar{d}''$, and $\beta\left(x,\bar{y},\bar{y}'',\bar{a}\right)$
isolates a complete type over $A$. Then $\exists\bar{y}''\beta\left(x,\bar{y},\bar{y}'',\bar{a}\right)\in\tp\left(c\bar{d}/A\right)$
 isolates a complete type, so $B\cup\left\{ c\right\} $ is atomic.
\end{proof}
\begin{rem}
In the stable case, our methods allow us to construct only a locally
atomic model (without Martin's Axiom, but still assuming that the
underlying language is countable). This is a classical result by Lachlan,
see \cite{MR0327505}, later improved upon by Newelski \cite{MR1067224},
where he replaces the assumption that the theory is countable by a
weaker one. 
\end{rem}
By Fact \ref{fac:tt has atomic}, the remark after Definition \ref{def:omega stable},
Example \ref{exa:superstable-continuum plus} and Theorem \ref{thm:atomic model MA superstable}
we conclude:
\begin{cor}
\label{cor:final corollary classification} Suppose $T$ is a countable
first order theory, $A\subseteq M\models T$ and $\kappa$ is a cardinal
such that the isolated types in $T\left(A\right)$ are dense and $\left|A\right|\leq\kappa^{+}$.
Then:
\begin{enumerate}
\item If $T$ is $\omega$-stable then $T\left(A\right)$ has an atomic
model.
\item If $T$ is superstable and $\MA[\left(\mbox{Cohen}\right)]{\kappa}$
holds then $T\left(A\right)$ has an atomic model. 
\item If $T$ is superstable but $\MA[\left(\mbox{Cohen}\right)]{\kappa}$
does not hold then $T\left(A\right)$ may not have an atomic model. 
\end{enumerate}
\end{cor}
\begin{problem}
Does Theorem \ref{thm:atomic model MA superstable} hold when $T$
is stable?
\end{problem}

\section{\label{sec:On-linear-orders}On linear orders}

In this section we will try to focus on the particular case when the
structures involved are linear orders. Assume that $L=\left\{ <\right\} $
(where $<$ is a binary relation symbol). Let $P$ be a forcing notion,
and $\tauname$ a $P$-name for an infinite linear order with universe
$\omega$. Again we ask: suppose that $\left(P,L,\tauname\right)$
has the isomorphism property (see Definition \ref{def:Isomorphism property}).
Does it follow that for some linear order $I=\left(X,<\right)\in\Vv$,
$P\forces\check{I}\cong\tauname$? 

In private communications with Zapletal, we were told that he solved
the problem positively for set theoretic trees (i.e., structures of
the form $\left(X,<\right)$ where $<$ is a partial order and the
set below every element is well-ordered). 

In Section \ref{sub:the class K mu} we translate the problem to that
of finding an atomic model in certain classes of theories, contained
in a bigger class $\weakLin{\mu}$. In Section \ref{sub:More-on-Ku+,Ku*}
we will further analyze these classes and give an equivalent definition.
In Section \ref{sub:An-example-of}, we give an example of a theory
in $\weakLin{\mu}$ for $\mu=\left(2^{\aleph_{0}}\right)^{+}$, with
no atomic model.

\subsection{\label{sub:the class K mu}The class $\protect\weakLin{\mu}$}

Let $\collaps{\mu}{\aleph_{0}}$ be the Levy collapse of $\mu$ to
$\aleph_{0}$. Note that as this forcing notion is weakly homogeneous,
for every first order sentence $\psi$ in the language of set theory
with parameters from $\Vv$, $\Vv\left[G_{1}\right]\models\psi$ iff
$\Vv\left[G_{2}\right]\models\psi$ for any two generic sets $G_{1}$
and $G_{2}$. Hence we may write $\Vv^{\collaps{\mu}{\aleph_{0}}}\models\psi$
meaning $\Vv\left[G\right]\models\psi$ for any generic $G$. See
\cite[Proposition 10.19]{MR1994835}. 
\begin{defn}
\label{def:K1}For a cardinal $\mu$, let $L_{\mu}=\left\{ <\right\} \cup\set{P_{i}}{i\in\mu}\cup\set{Q_{j}}{j\in\mu}$
where $P_{i}$ are unary predicates and $Q_{j}$ binary relation symbols.
let $\weakLin{\mu}$ be the class of complete first order theories
$T$ in a language $L_{T}=\left\{ <\right\} \cup\set{P_{i}}{i\in u_{T}}\cup\set{Q_{j}}{j\in v_{T}}\subseteq L_{\mu}$
equipped with a function $f_{T}:v^{2}\to v$ such that:
\begin{itemize}
\item The theory $T$ says that $<$ is a linear order; the sets $P_{i}$
are nonempty and disjoint; the sets $Q_{j}$ are nonempty and disjoint;
$Q_{j}\left(x,y\right)\vdash x<y$; \emph{additivity} of $T$: for
all $x<y<z$, $Q_{j_{1}}\left(x,y\right)\land Q_{j_{2}}\left(y,z\right)\vdash Q_{f_{T}\left(j_{1},j_{2}\right)}\left(x,z\right)$;
the formula $P_{i}\left(x\right)$ isolates a complete type for $i\in u_{T}$;
the formula $Q_{j}\left(x,y\right)$ isolates a complete type for
$j\in v_{T}$; the isolated types are dense in $T$, and for any (equivalently,
some) countable atomic model $M$ of $T$ in $\Vv^{\collaps{\mu}{\aleph_{0}}}$,
the sets $P_{i}^{M}$ and $Q_{j}^{M}$ form a partition of $M$, $<^{M}$
respectively. 
\end{itemize}
\end{defn}
\begin{rem}
\label{rem:some equivalent definitions of K1}Suppose that $T\in\weakLin{\mu}$.
Then for every $n<\omega$, if a formula of the form $\bigwedge_{i<n}Q_{j_{i}}\left(x_{i},x_{i+1}\right)$
is consistent, then it isolates a complete type. Indeed, it is enough
to show that this is true in a countable atomic model $M$ of $T$
in $\Vv^{\collaps{\mu}{\aleph_{0}}}$. So let $\bar{x}=\left(x_{0},\ldots,x_{n}\right)$
and $\bar{y}=\left(y_{0},\ldots,y_{n}\right)$ be finite increasing
tuples from $M$ such that $Q_{j_{i}}^{M}\left(x_{i},x_{i+1}\right)$
and $Q_{j_{i}}^{M}\left(y_{i},y_{i+1}\right)$ hold for all $i<n$.
Then, as $Q_{j_{i}}$ isolate complete types, and as atomic models
are homogeneous, there are automorphisms $f_{i}$ of $M$ which take
$x_{i},x_{i+1}$ to $y_{i},y_{i+1}$. Let 
\[
f=f_{0}\upharpoonright\left(-\infty,x_{1}\right]\cup f_{1}\upharpoonright\left(x_{1},x_{2}\right]\cup\ldots\cup f_{n-1}\upharpoonright\left(x_{n-1},\infty\right).
\]
Then, as the $Q_{j}$'s form a partition of $<^{M}$, it follows that
$f$ is an automorphism of $M$ by the additivity of $T$. It follows
that these types are dense in $T$. 

Now note that an equivalent definition to Definition \ref{def:K1}
can be obtained by replacing the additivity requirement, the requirement
that the isolated types are dense, and the requirement that the $P_{i}$'s
and $Q_{j}$'s form a partition of a countable atomic model (after
the collapse) by asking that formulas of the form $\bigwedge_{i<n}Q_{j_{i}}\left(x_{i},x_{i+1}\right)$,
if consistent, isolate a complete type, and that types of such forms,
as well as $P_{i}\left(x\right)$ are dense in $T$ (we should also
allow formulas of the form $P_{i}\left(x\right)\land x=y$, etc, to
be completely formal).  
\end{rem}
For notational simplicity it is useful to write $Q\left(a,b\right)=j$
for $a\neq b\in M\models T$ when $j\in v_{T}$ is the unique element
such that $\left(a,b\right)\in Q_{j}^{M}$ or $\left(b,a\right)\in Q_{j}^{M}$,
and similarly $P\left(a\right)=i$ if $a\in P_{i}^{M}$.

We define two special subclasses of $\weakLin{\mu}$.
\begin{defn}
\label{def:Kmu*, K_mu+}Let $\fullLin{\mu}$ be the class of theories
$T\in\weakLin{\mu}$ such that for any $j\in v_{T}$ there is an $L_{\mu^{+},\omega}$-formula
$\psi_{j}\left(x,y\right)$ in the language $\set{P_{i}}{i\in u_{T}}\cup\left\{ <\right\} $
such that for any countable atomic model $M$ of $T$ in $\Vv^{\collaps{\mu}{,\aleph_{0}}}$,
$\psi_{j}^{M}=Q_{j}^{M}$. (After the collapse, $\psi_{j}$ is an
$L_{\omega_{1},\omega}$-formula.)

Let $\interLin{\mu}$ be the class of theories $T\in\weakLin{\mu}$
such that for any $j\in v_{T}$ there is an $L_{\mu^{+},\omega}$-formula
$\psi_{j}\left(x,y\right)$ in the language $\left\{ <\right\} $
such that for any countable atomic model $M$ of $T$ in $\Vv^{\collaps{\mu}{,\aleph_{0}}}$,
$\psi_{j}^{M}=Q_{j}^{M}$. 
\end{defn}
As usual, this definition does not depend on the choice of atomic
model or the generic extension. Note that $\weakLin{\mu}\supseteq\fullLin{\mu}\supseteq\interLin{\mu}$. 
\begin{rem}
\label{rem:in Kmu*, P_i are also formulas}If $T\in\interLin{\mu}$,
then for every $i\in u_{T}$, there is a formula $\varphi_{i}$ in
$L_{\mu^{+},\omega}$ such that in any countable atomic model $M$
of $T$ in $\Vv^{\collaps{\mu}{,\aleph_{0}}}$, $\varphi_{i}^{M}=P_{j}^{M}$.
Why? Suppose that $M$ is a countable atomic model of $T$ in $\Vv^{\collaps{\mu}{\aleph_{0}}}$.
Let $a\in P_{i}^{M}$, and let $\ScottPsi_{a}\left(x\right)$ be the
Scott formula which isolates the complete $L_{\omega_{1},\omega}$-type
of $a$ in $\left\{ <\right\} $ (in the context of Proposition \ref{prop:Scott},
this is $\phi_{\beta,a}$). As $P_{i}$ isolates a complete type,
this formula does not depend on $a$, so we define $\varphi_{i}=\ScottPsi_{a}$.
In addition, $\varphi_{i}$ does not depend on $M$, and hence by
\ref{cor:product equal},$\varphi_{i}$ is in $\Vv$. By definition,
$P_{i}^{M}\subseteq\varphi_{i}^{M}$ for any such $M$. It remains
to show that if $a\models\varphi_{i}$ in some such $M$, and $a\in P_{i'}^{M}$
then $i'=i$. If not, by choice of $\varphi_{i}$, there is some $b\models\varphi_{i}$
in $P_{i}^{M}$, and suppose that $a<b$. Let $j\in v_{T}$ be such
that $\left(a,b\right)\in Q_{j}^{M}$. Since$a$ and $b$ satisfy
$\varphi_{i}$, there is an automorphism $\sigma$ of $M\upharpoonright\left\{ <\right\} $
taking $a$ to $b$. Then by the definition of $\interLin{\mu}$,
$\left(\sigma\left(a\right),\sigma\left(b\right)\right)\in Q_{j}^{M}$.
However $Q_{j}$ isolates a complete type, so it must be that $Q_{j}\left(x,y\right)\vdash P_{i'}\left(x\right)$
and hence $b=\sigma\left(a\right)\in P_{i'}^{M}$ so $i=i'$. 
\end{rem}
The following theorem translates the question of finding a linear
order with the isomorphism property which is not realized in $\Vv$
to the question of finding a theory in $\fullLin{\mu}$ without an
atomic model, as in Section \ref{sec:Translating-the-question}. It
seems that in order to produce a counterexample, $\fullLin{\mu}$
allows more freedom than $\interLin{\mu}$ (as it contains more theories),
but in the end they are equivalent in this sense. 
\begin{thm}
\label{thm:translation order}Let $\mu$ be some cardinal. The following
are equivalent:
\begin{enumerate}
\item For any forcing notion $P$ such that $P\not\forces\left|\check{\left(\mu^{+}\right)}\right|<\omega_{1}$
 and any $P$-name $\tauname$ such that $P$ forces $\tauname$
to be a linear order on $\omega$, if $\left(P,\left\{ <\right\} ,\tauname\right)$
has the isomorphism property, then for some linear order $I=\left(X,<\right)\in\Vv$,
$P\forces\check{I}\cong\tauname$. 
\item Every theory $T\in\fullLin{\mu}$ has an atomic model of size $\mu$. 
\item Every theory $T\in\interLin{\mu}$ has an atomic model of size $\mu$. 
\end{enumerate}
\end{thm}
\begin{proof}
(1) $\Rightarrow$ (2): Suppose that $T\in\fullLin{\mu}$. Let $P=\collaps{\mu}{\aleph_{0}}$
be the Levy collapse of $\mu$ to $\aleph_{0}$. As the isolated types
are dense in $T$, $P$ forces that $T$ has a countable atomic model,
so let $\dot{M}$ be a name for it. Now let us work in $\Vv^{P}$,
and let $M\in\Vv\left[G\right]$ be a countable atomic model of $T$
with universe $\omega$. 

Let us introduce the following notation. For a family $\set{\left(X_{i},<_{i}\right)}{i\in I}$
of linear orders, where $I$ is linearly ordered by $\prec$, we let
$\sum_{i\in I}X_{i}$ be the linear order whose universe is the disjoint
union of the sets $X_{i}$ and the order $<$ is such that $\mathordi <\upharpoonright X_{i}=\mathordi{<_{i}}$
and for $i\neq j$, $a\in X_{i}$ and $b\in X_{j}$, $a<b$ iff $i\prec j$.

Denote by $\Qq$ the usual dense linear order on the rational numbers.
Define the linear order $X_{M}=\left(X,<\right)$ as the sum of linear
orders $\sum_{a\in M}X_{a}$ where for each $a\in M$, $X_{a}=\Qq+\left(P\left(a\right)+2\right)$
(recall that $P\left(a\right)$ is the unique $i\in u_{T}$ --- now
a countable ordinal --- such that $a\in P_{i}$. It is well defined
because in an atomic model, the $P_{i}$'s form a partition). For
instance, if $a\in P_{0}$, then $X_{a}=\Qq+\left\{ 0,1\right\} $. 

Let $\tauname$ be a $P$-name for $X_{\dot{M}}$. First we claim
that $\left(P,\left\{ <\right\} ,\tauname\right)$ has the isomorphism
property. This follows easily from the fact that after the collapse,
if $M\cong N$ are two countable atomic models of $T$, then $X_{M}\cong X_{N}$
(as linear orders). Now, $P\x P\forces\dot{M}_{1}\cong\dot{M}_{2}$
(where, as usual, $\dot{M}_{1}$ is the $P\x P$-name $\dot{M}\left[G_{1}\right]$
where $G_{1}\x G_{2}$ is a generic for $P\x P$, etc.), so in $\Vv^{P\x P}$,
$\tauname_{1}\cong\tauname_{2}$ as desired. 

By (1), there is some linear order $I=\left(X,<\right)\in\Vv$ such
that $P\forces\check{I}\cong\tauname$. Now we want to recover an
atomic model of $T$ from $I$. 

Let $X_{0}\subseteq X$ be the set of all $x\in X$ with a densely
ordered open neighborhood without endpoints, and let $X_{1}=X\backslash X_{0}$.
For each $x\in X_{1}$, there are unique $x_{0}=x_{0}\left(x\right)\leq x\leq x_{1}\left(x\right)=x_{1}$
such that:
\begin{itemize}
\item $x_{0}<x_{1}$; $x_{0},x_{1}\in X_{1}$; every point in $\left[x_{0},x_{1}\right)$
has a successor; for every $z<x_{0}$ there is some $y\in\left(z,x_{0}\right)\cap X_{0}$,
and similarly, for every $x>x_{1}$ there is some $y\in\left(x_{1},x\right)\cap X_{0}$.
\end{itemize}
Why? This sentence is absolute and since it is true after the collapse,
it is also true in $\Vv$. We now know that after the collapse, the
closed interval $\left[x_{0},x_{1}\right]$ is well-ordered and has
the same order type as $i+2$ for some countable ordinal $i\in u_{T}$,
so in $\Vv$ the same is true (except that now $i<\mu^{+}$). 

There is a natural convex equivalence relation $\sim$ on $X_{1}$:
two points are equivalent if they define the same $x_{0}$ and $x_{1}$.
Let $Y=X_{1}/\sim$, and define an $L_{T}$-structure $N$ with universe
$Y$ such that 
\[
P_{i}^{N}=\set{\left[x\right]_{\sim}\in Y}{\otp\left[x_{0}\left(x\right),x_{1}\left(x\right)\right]=i+2},
\]
and such that $Q_{j}^{N}$ is defined using the $L_{\mu^{+},\omega}$-formula
promised in Definition \ref{def:Kmu*, K_mu+}. 

Now, by absoluteness of the construction, $N$ is also the result
of this construction in $\Vv^{P}$, and there, as $I\cong X_{M}$,
we get that $M\upharpoonright\left\{ <\right\} \cup\set{P_{i}}{i\in u_{T}}$
is isomorphic to $N\upharpoonright\left\{ <\right\} \cup\set{P_{i}}{i\in u_{T}}$,
and hence $N\cong M$ and we are done. 

(2) $\Rightarrow$ (3): Immediate as $\interLin{\mu}\subseteq\fullLin{\mu}$. 

(3) $\Rightarrow$ (1): Suppose that $\left(P,\left\{ <\right\} ,\tauname\right)$
has the isomorphism property. Let $G$ be a generic set for $P$,
and let $M=\tauname\left[G\right]$. For each $a\in M$, let $\ScottPsi_{a}\left(x\right)$
be as in Remark \ref{rem:in Kmu*, P_i are also formulas} and similarly
define $\ScottPsi_{a,b}\left(x,y\right)$ for $a<b$ in $M$. Let
${\cal U}=\set{\ScottPsi_{a}}{a\in M}$ and let ${\cal V}=\set{\ScottPsi_{a,b}}{a<b\in M}$.
By Corollary \ref{cor:product equal}, ${\cal U}$ and ${\cal V}$
are in $\Vv$ (although here they are sets of $L_{\mu^{+},\omega}$-formulas)
and their size is at most $\mu$ by the assumption on $P$. Enumerate
(in $\Vv$) ${\cal U}=\set{\psi_{i}}{i\in u}$ and ${\cal V}=\set{\xi_{j}}{j\in v}$
where $u,v$ are subsets of $\mu$. 

Let $L=\left\{ <\right\} \cup\set{P_{i}}{i\in u}\cup\set{Q_{j}}{j\in v}$.
Let $M'$ be an expansion of $M$ to $L$ defined by $P_{i}^{M'}=\set{x\in M}{x\models\psi_{i}}$
and $Q_{j}^{M'}=\set{\left(x,y\right)\in\mathordi{<^{M}}}{\left(x,y\right)\models\xi_{j}}$.
Note that $M'$ is atomic and that formulas of the form $\bigwedge_{i<n}Q_{j_{i}}\left(x_{i},x_{i+1}\right)$,
if consistent, isolate a complete type, and the same is true for formulas
of the form $P_{i}\left(x\right)$ (this follows from the fact that
any automorphism of $M$ is also an automorphism of $M'$). 

Let $T=Th\left(M'\right)$. By Corollary \ref{cor:product equal},
$T\in\Vv$. We claim that $T\in\interLin{\mu}$. By Definition \ref{def:K1}
and Remark \ref{rem:some equivalent definitions of K1}, $T\in\weakLin{\mu}$
and $T\in\interLin{\mu}$ as witnessed by $\xi_{j}$. By (3), $T$
has an atomic model $N'\in\Vv$. Hence in $\Vv\left[G\right]$, $N'\cong M'$.
It follows that $N\upharpoonright\left\{ <\right\} $ is isomorphic
to $M$ in $\Vv\left[G\right]$.
\end{proof}

\subsection{\label{sub:More-on-Ku+,Ku*}More on $\protect\interLin{\mu}$ and
$\protect\fullLin{\mu}$}
\begin{defn}
\label{def:the equivalence relation}Suppose that $T\in\weakLin{\mu}$.
Suppose that $E$ is some equivalence relation on $v_{T}$. By induction
on $\alpha<\mu^{+}$, define equivalence relations $\eqT E{\alpha}$
as follows:

For $\alpha=0$, $\eqT E0=E$. 

For $\alpha$ limit, $j_{1}\eqT E{\alpha}j_{2}$ iff for all $\beta<\alpha$,
$j_{1}\eqT E{\beta}j_{2}$.

For $\alpha=\beta+1$, $j_{1}\eqT E{\alpha}j_{2}$ iff $j_{1}\eqT E{\beta}j_{2}$
and there exists a countable atomic model $M\in\Vv^{\collaps{\mu}{\aleph_{0}}}$
of $T$ such that:
\begin{itemize}
\item For any $\left(a,b\right)\in Q_{j_{1}}^{M}$, $\left(a',b'\right)\in Q_{j_{2}}^{M}$,
and for any $c\in M\backslash\left\{ a,b\right\} $ there is some
$c'\in M\backslash\left\{ a',b'\right\} $ such that $abc$ and $a'b'c'$
have the same order type, and $Q\left(a,c\right)\eqT E{\beta}Q\left(a',c'\right)$
and $Q\left(c,b\right)\eqT E{\beta}Q\left(c',b'\right)$. 
\item The same, replacing $j_{1}$ with $j_{2}$. 
\end{itemize}
\end{defn}
Note that reflexivity of $\eqT E{\alpha}$ follows from the fact that
$Q_{j}\left(x,y\right)$ isolates a complete type in $T$ (and from
the fact that a countable atomic model is homogeneous).

Let us say that an equivalence relation $E$ on $v_{T}$ is \emph{additive}
if for all $j_{1},j_{2}$ and $j_{1}',j_{2}'$ in $v_{T}$ such that
$Q_{j_{1}}\left(x,y\right)\land Q_{j_{2}}\left(y,z\right)$ and $Q_{j_{1}'}\left(x,y\right)\land Q_{j_{2}'}\left(y,z\right)$
are both consistent, $j_{1}\mathrela Ej_{1}'\land j_{2}\mathrela Ej_{2}'\Rightarrow f_{T}\left(j_{1},j_{2}\right)\mathrela Ef_{T}\left(j_{1}',j_{2}'\right)$.
Say that $E$ is \emph{edge preserving} if whenever $x<y,x'<y'$ and
$Q\left(x,y\right)\mathrela EQ\left(x',y'\right)$ then $\tp\left(x/\emptyset\right)=\tp\left(x'/\emptyset\right)$
and $\tp\left(y/\emptyset\right)=\tp\left(y'/\emptyset\right)$. 
\begin{claim}
\label{claim:some properties of EqT}Suppose that $E$ is some equivalence
relation on $v_{T}$, where $T\in\weakLin{\mu}$.
\begin{enumerate}
\item For all $\alpha<\mu^{+}$, $\mathordi{\eqT E{\alpha+1}}\subseteq\mathordi{\eqT E{\alpha}}$.
Hence for some $\alpha=\alpha\left(T,E\right)$, $\mathordi{\eqT E{\alpha+1}}=\mathordi{\eqT E{\alpha}}$
. 
\item Definition \ref{def:the equivalence relation} does not depend on
the choice of a generic set $G$ for $\collaps{\mu}{\aleph_{0}}$
because $\collaps{\mu}{\aleph_{0}}$ is weakly homogeneous (see above).%

\item If $E$ is additive, then so is $\eqT{\alpha}E$ for any $\alpha<\mu^{+}$. 
\item In Definition \ref{def:the equivalence relation}, we can do any of
the following changes and get an equivalent definition:

\begin{enumerate}
\item Replace ``there exists a countable atomic model'' (which exists
since the isolated types are dense), by ``for any countable atomic
model'' (as any two are isomorphic). 
\item Replace the bullets by ``for some $\left(a,b\right)\in Q_{j_{1}}^{M}$,
$\left(a',b'\right)\in Q_{j_{2}}^{M}$ the following holds. For any
$c\in M\backslash\left\{ a,b\right\} $, there is some $c'\in M\backslash\left\{ a',b'\right\} $
such that $abc$ and $a'b'c'$ have the same order type, $Q\left(a,c\right)\eqT E{\beta}Q\left(a',c'\right)$
and $Q\left(c,b\right)\eqT E{\beta}Q\left(c',b'\right)$ and vice
versa.'' 
\item If $E$ is edge preserving and additive, then we can replace $M\backslash\left\{ a,b\right\} $
and $M\backslash\left\{ a',b'\right\} $ by the intervals $\left(a,b\right)$
and $\left(a',b'\right)$ in the bullets.
\end{enumerate}
\end{enumerate}
\end{claim}
\begin{proof}
(1) is clear, (2) is explained in the claim. For (3) use induction
on $\alpha$ and the fact that $T$ is additive. (4) (a) is explained
above, (4) (b) follows by the fact that $Q_{j}$ isolates a complete
type and by homogeneity, (4) (c) is proved by induction on $\alpha$
using (3). 
\end{proof}
As usual, assume that $T\in\weakLin{\mu}$. Given an equivalence relation
$E$ on $v_{T}$, say that $E$ is \emph{definable} in some sub-language
$\left\{ <\right\} \subseteq L'\subseteq L_{T}$ if for every $E$-class
$C\subseteq v_{T}$ there is some $L_{\mu^{+},\aleph_{0}}$-formula
$\psi_{C}$ in $L'$ such that if $M\in\Vv^{\collaps{\mu}{\aleph_{0}}}$
is a countable atomic model of $T$, then $\psi_{C}^{M}=C^{M}=\bigcup\set{Q_{j}^{M}}{j\in C}$.
\begin{prop}
\label{prop:definition of Q_j}Suppose that $T\in\weakLin{\mu}$ and
$E$ is some equivalence relation on $v_{T}$. Assume that $E$ is
definable in some sub-language $\left\{ <\right\} \subseteq L'\subseteq L$.
Then for every $\alpha<\mu^{+}$, $\eqT E{\alpha}$ is also definable
in $L'$. \end{prop}
\begin{proof}
The proof is by induction on $\alpha<\mu^{+}$. For $\alpha=0$ this
is given. 

Suppose $\alpha>0$ is a limit. Then for each $\eqT E{\alpha}$-class
$C$ there is a sequence $\sequence{D_{\beta}}{\beta<\alpha}$ of
$\eqT E{\beta}$-classes, such that $C=\bigcap_{\beta<\alpha}D_{\beta}$,
and for any countable atomic model $M\in\Vv^{\collaps{\mu}{\aleph_{0}}}$,
$\left(x,y\right)\in C^{M}$ iff $\left(x,y\right)\in D_{\beta}^{M}$
for all $\beta<\alpha$. 

Suppose $\alpha=\beta+1$. Then for each $\eqT E{\alpha}$-class $C$
there is an $\eqT E{\beta}$-class $D$ and a set $A$ of tuples of
the form $\left(r,D_{1},D_{2}\right)$ where $r$ is an order type
of three points, and $D_{1}$, $D_{2}$ are $\eqT E{\beta}$-classes
such that for any countable atomic $M\in\Vv^{\collaps{\mu}{\aleph_{0}}}$,
$\left(x,y\right)\in C^{M}$ iff $\left(x,y\right)\in D^{M}$ and:

\begin{itemize}
\item For every $z\in M\backslash\left\{ x,y\right\} $, there is some triple
$\left(r,D_{1},D_{2}\right)\in A$ such that $\otp\left(xyz\right)=r$,
$\left(x,z\right)\in D_{1}^{M}$ (if $x<z$, otherwise we ask that
$\left(z,x\right)\in D_{1}^{M}$) and $\left(z,y\right)\in D_{2}^{M}$
(if $z<y$, else same as before) and for every triple $\left(p,E_{1},E_{2}\right)\in A$
there is some $w\in M\backslash\left\{ x,y\right\} $ such that $p=\otp\left(xyw\right)$,
$\left(x,w\right)\in E_{1}^{M}$ (if $x<w$, else see above) and $\left(w,y\right)\in E_{2}^{M}$
(if $w<y$, else see above). 
\end{itemize}
Note that these parameters ($D$, $A$) belong to $\Vv$ and depend
only on $C$ and not and the choice of atomic model or generic set.
By induction this can be written in $L_{\mu^{+},\aleph_{0}}$ (as
there are at most $\mu$-many classes) using $L'$ (since $\mathordi <\in L'$).
\end{proof}
Let $T\in\weakLin{\mu}$. Let $E_{0}$ be the following equivalence
relation on $v_{T}$: $j_{1}\mathrela{E_{0}}j_{2}$ iff for some $i_{1},i_{2}\in u_{T}$,
$Q_{j_{1}}\left(x,y\right)\vdash P_{i_{1}}\left(x\right)\land P_{i_{2}}\left(y\right)$
and $Q_{j_{2}}\left(x,y\right)\vdash P_{i_{1}}\left(x\right)\land P_{i_{2}}\left(y\right)$. 

Let $E_{1}$ be the trivial equivalence relation on $v_{T}$, i.e.,
$E_{1}=v_{T}\x v_{T}$. 
\begin{claim}
\label{claim:E0 and E1 are good}Suppose $T\in\weakLin{\mu}$. The
relations $E_{0}$ and $E_{1}$ are definable in the languages $L_{0}=\left\{ <\right\} \cup\set{P_{i}}{i\in u_{T}}$
and $L_{1}=\left\{ <\right\} $ respectively and are additive. The
relation $E_{1}$ is edge preserving. \end{claim}
\begin{proof}
The first assertion is easy to check. The second follows from the
fact that $P_{i}\left(x\right)$ isolates a complete type. 
\end{proof}
As $E_{0}\subseteq E_{1}$, it follows that $\mathordi{\eqT{E_{0}}{\alpha}}\subseteq\mathordi{\eqT{E_{1}}{\alpha}}$
for all $\alpha<\mu^{+}$.
\begin{rem}
We can also define ``intermediate'' equivalence relations between
$E_{0}$ and $E_{1}$, taking into account a specific subset $s$
of $u_{T}$, and then define $j_{1}\mathrela{E_{s}}j_{2}$ iff $P\left(a\right)\equiv_{s}P\left(a'\right)$
and $P\left(b\right)\equiv_{s}P\left(b'\right)$ whenever $\left(a,b\right)\in Q_{j_{1}}$
and $\left(a',b'\right)\in Q_{j_{2}}$ where $\equiv_{s}$ is the
equivalence relation on $u_{T}$ whose classes are $\left\{ s\right\} \cup\set{\left\{ i\right\} }{i\in u_{T}\backslash s}$.
In this notation $E_{1}=E_{u_{T}}$ and $E_{0}=E_{\emptyset}$. These
are less important but worth mentioning. \end{rem}
\begin{prop}
\label{prop: equivalent definition of Kmu}Let $\mu$ be a cardinal.
Then $\fullLin{\mu}$ is the class of theories $T\in\weakLin{\mu}$
such that $\eqT{E_{0}}{\alpha\left(T,E_{0}\right)}$ is equality (see
Claim \ref{claim:some properties of EqT} (1)). Similarly, $\interLin{\mu}$
is the class of theories $T\in\weakLin{\mu}$ such that $\eqT{E_{1}}{\alpha\left(T,E_{1}\right)}$
is equality.\end{prop}
\begin{proof}
The proofs for $\fullLin{\mu}$ and $\interLin{\mu}$ are similar,
so we will do the $\interLin{\mu}$ case.

Suppose that $T\in\interLin{\mu}$ as witnessed by $\psi_{j}$ for
$j\in v_{T}$. Suppose $j_{1}\eqT{E_{1}}{\alpha\left(T,E_{1}\right)}j_{2}$.
Let $M$ be a countable atomic model of $T$ after the collapse, and
let $\left(a,b\right)\in Q_{j_{1}}^{M}$, $\left(a',b'\right)\in Q_{j_{2}}^{M}$.
Let $F$ be the set of all finite order-preserving partial functions
$g$ from $M$ to $M$ which map $a,b$ to $a',b'$ and such that
for $x<y$ in its domain, $Q\left(x,y\right)\eqT{E_{1}}{\alpha\left(T,E_{1}\right)}Q\left(g\left(x\right),g\left(y\right)\right)$.
Then $F$ is a back and forth system by the choice of $\alpha\left(T,E_{1}\right)$
and by the additivity of $\eqT{E_{1}}{\alpha}$ (see Claim \ref{claim:some properties of EqT}
(3)) and hence there is an automorphism of $M\upharpoonright\left\{ <\right\} $
taking $a,b$ to $a',b'$, so both $\left(a,b\right)$ and $\left(a',b'\right)$
satisfy $\psi_{j_{1}}$ and $\psi_{j_{2}}$. But the $Q_{j}$'s are
disjoint, and $\psi_{j}$ defines $Q_{j}$ in $M$, so $j_{1}=j_{2}$. 

Conversely, by Proposition \ref{prop:definition of Q_j}, for every
$j\in v_{T}$, there is some formula $\psi_{j}$ in $L_{\mu^{+},\omega}$
which defines (in the sense of said proposition) the class of $j$
in $E_{1,\alpha\left(T,E_{1}\right)}$ which is just $\left\{ j\right\} $. 
\end{proof}

\subsection{\label{sub:An-example-of}An example of a theory in $\protect\weakLin{\mu}$
without an atomic model}
\begin{thm}
Let $\mu=\left(2^{\aleph_{0}}\right)^{+}$. There is some $T\in\weakLin{\mu}$
without an atomic model.\end{thm}
\begin{proof}
Let $\Qq^{+}$ be the set of positive rationals, and let $v_{T}=\set{\left(i_{1},i_{2},\left(n,q\right)\right)}{i_{1},i_{2}\in\mu,n<\omega,q\in\Qq^{+}}$.
Let $L=L_{T}=\left\{ <\right\} \cup\set{P_{i}}{i\in\mu}\cup\set{Q_{j}}{j\in v_{T}}.$ 

Let $P=\collaps{\mu}{\aleph_{0}}$, and let $G$ be a generic set.
Work in $\Vv\left[G\right]$, where $\mu$ is countable. Let $M$
be the following structure. Its universe $X$ is the set of all functions
$\eta$ from $\omega$ to $\Qq$ such that for some $n<\omega$, $\eta\left(m\right)=0$
for all $m\geq n$ (so that $X$ is countable). The order is: $\eta_{1}>\eta_{2}$
iff for $n=\min\set k{\eta_{1}\left(k\right)\neq\eta_{2}\left(k\right)}$,
$\eta_{1}\left(n\right)>\eta_{2}\left(n\right)$. As an order, this
is just a dense linear order with no endpoints. 

Choose $P_{i}^{M}$ in a dense way. More precisely, choose $P_{i}^{M}$
in such a way that each $P_{i}^{M}$ is dense and unbounded from above
and below, and such that the $P_{i}$'s form a partition of $M$.
It is easy to see that this is possible to construct ``by hand'',
or one can see this as the countable atomic model of the model completion
of the theory of linear order and infinitely many colors. 

Now we must define $Q_{i_{1},i_{2},\left(n,q\right)}$ where $i_{1},i_{2}\in\mu$,
$n\in\omega$ and $q\in\Qq^{+}$. Suppose $\eta_{1}<\eta_{2}$, then
$\left(\eta_{1},\eta_{2}\right)\in Q_{i_{1},i_{2},\left(n,q\right)}^{M}$
iff $\eta_{1}\in P_{i_{1}}^{M}$, $\eta_{2}\in P_{i_{2}}^{M}$, $n=\min\set k{\eta_{1}\left(k\right)\neq\eta_{2}\left(k\right)}$
and $q=\eta_{2}\left(n\right)-\eta_{1}\left(n\right)$ (which must
be $>0$). Note that the $Q_{j}$'s form a partition of $<^{M}$. 

As above, write $P^{M}\left(\eta\right)=i$ to denote that $\eta\in P_{i}^{M}$
and similarly $Q^{M}\left(\eta_{1},\eta_{2}\right)=j$. 
\begin{claim}
The theory of $M$ is additive: given $j_{1},j_{2}\in v_{T}$ there
is a unique $j_{3}=f_{Th\left(M\right)}\left(j_{1},j_{2}\right)$
such that $Q_{j_{1}}\left(x,y\right)\land Q_{j_{2}}\left(y,z\right)\vdash Q_{j_{3}}\left(x,z\right)$. \end{claim}
\begin{proof}
[Proof of claim] There are several cases to check. Suppose that $j_{1}=\left(i_{1},i_{2},\left(n,q\right)\right)$,
$j_{2}=\left(k_{1},k_{2},\left(m,r\right)\right)$. Then, if $n>m$,
then $j_{3}=\left(i_{1},k_{2},\left(m,r\right)\right)$, and if $n<m$,
$j_{3}=\left(i_{1},k_{2},\left(n,q\right)\right)$. If $n=m$, then
$j_{3}=\left(i_{1},k_{2},\left(n,q+r\right)\right)$. \end{proof}
\begin{claim}
\label{claim:isomorphism property}Suppose that $N$ is an $L$-structure
(perhaps in some transitive model of set theory containing $\Vv^{P}$)
with universe $X$ such that $<^{N}=<^{M}$, the sets $P_{i}^{N}$
form a partition of $X$ for $i\in\mu$ such that $P_{i}^{N}$ is
dense and unbounded, and $Q_{j}^{N}$ defined in the same way as $Q_{j}^{M}$
for $j\in v_{T}$. Then $M\cong N$. Moreover, $M$ is atomic, and
the formulas of the form $P_{i}\left(x\right)$, $\bigwedge_{i<n}Q_{j_{i}}\left(x_{i},x_{i+1}\right)$
isolate complete types which are dense in $Th\left(M\right)$. \end{claim}
\begin{proof}
[Proof of claim] Note that $Th\left(N\right)$ is also additive,
and that $f_{Th\left(N\right)}=f_{Th\left(M\right)}$. We do a back
and forth argument. Suppose that $g:M\to N$ is a partial finite isomorphism
from some finite subset $s\subseteq X$ to $g\left(s\right)$, and
we are given $\eta\in X$ which we want to add to its domain. Enumerate
$s=\set{\eta_{i}}{i<n}$ where $\eta_{0}<\ldots<\eta_{n-1}$, and
suppose $\eta_{i}<\eta<\eta_{i+1}$ (where $-1\leq i<n$ and $\eta_{-1}=\infty$,
$\eta_{n}=\infty$). Let $g\left(s\right)=\set{\nu_{i}}{i<n}$ where
$\nu_{0}<\ldots<\nu_{n-1}$. By additivity, it is enough to find some
$\nu\in X$ such that $\nu_{i}<\nu<\nu_{i+1}$, $P^{N}\left(\nu\right)=P^{M}\left(\eta\right)$,
$Q^{M}\left(\eta_{i},\eta\right)=Q^{N}\left(\nu_{i},\nu\right)$ and
$Q^{M}\left(\eta,\eta_{i+1}\right)=Q^{N}\left(\nu,\nu_{i+1}\right)$.
This follows easily by observing that for all $i<\mu$ and any $\eta\in X$
and $n<\omega$, there is some $\eta'\in X$ such that $\eta\upharpoonright n=\eta'\upharpoonright n$
and $P\left(\eta'\right)=i$: consider $\eta\upharpoonright n\concat\bar{0}$
(where $\bar{0}$ is just an infinite sequence of zeros), and $\eta\concat\left\langle 1\right\rangle \concat\bar{0}$.
Find $\eta'$ between these two. 

The moreover part follows from the previous paragraph, the fact that
$Q_{j}\left(x,y\right)\vdash P_{i_{1}}\left(x\right)\land P_{i_{2}}\left(y\right)$
for some $i_{1},i_{2}<\mu$ and the additivity.
\end{proof}
Let $\tauname$ be a $P$-name for $M$. By Claim \ref{claim:isomorphism property},
$\left(P,L,\tauname\right)$ has the isomorphism property, and hence
$T=Th\left(M\right)\in\Vv$. We also get that $T\in\weakLin{\mu}$.

\begin{claim}
The theory $T$ has no atomic model in $\Vv$. \end{claim}
\begin{proof}
[Proof of claim] Suppose $N$ is an atomic model of $T$. As $N$
is atomic, by the second claim, $\set{Q_{j}^{N}}{j\in v_{T}}$ partition
$<^{N}$. Define a coloring of increasing pairs, $c:\mathordi{<^{N}}\to\omega\x\Qq$
by $c\left(x,y\right)=\left(n,q\right)$ iff $Q^{N}\left(x,y\right)=\left(i_{1},i_{2},\left(n,q\right)\right)$
for some $i_{1},i_{2}<\mu$. By Erd\"os-Rado, for some infinite set
$A\subseteq N$, and some $n,q$, $c\left(x,y\right)=\left(n,q\right)$
for all $x<y$ in $A$. Since $A$ is infinite, we can find $x<y<z$
in $A$. But $T$ forbids a triple $x<y<z$ with 
\[
\left(n,q\right)=c\left(x,y\right)=c\left(y,z\right)=c\left(x,z\right)
\]
 --- a contradiction. 
\end{proof}
\end{proof}
\begin{problem}
What more can be said about theories from $\weakLin{\mu}$, $\fullLin{\mu}$
and $\interLin{\mu}$? When do they have atomic models? We saw that
for $\mu=\left(2^{\aleph_{0}}\right)^{+}$, there is a theory in $\weakLin{\mu}$
without an atomic model. Is the same true for some $\mu$ and $\fullLin{\mu}$
(equivalently, by Theorem \ref{thm:translation order}, $\interLin{\mu}$)? \end{problem}
\begin{rem}
\label{rem:AfterRelease}After this paper has appeared online, it
came to our attention that a recent paper by Knight, Montalban and
Schweber \cite{Knight2014} deals with similar notions and proves
some similar results, though for different purposes and with different
methods. Their notion of a generically presentable structure is very
close to our isomorphism property (see Definition \ref{def:Isomorphism property}).
There is some overlapping between the two papers (for instance our
Corollary \ref{cor:not-collaps-ok} and their Theorem 3.12). Their
proofs also use Scott sentences, but our approach is different (they
use Fra\"{i}ss\'{e} limits and we use atomic models), and the focuses
of the two papers are completely different. Two other recent  papers,
one by Baldwin, Friedman, Koerwien and Laskowski \cite{BaldwinFriedmanLas},
and another by Larson \cite{Larson2014} have some overlapping with
\cite{Knight2014} (they all give a new proof of a result of Harrington
regarding Vaught's conjecture using different methods). 
\end{rem}
\bibliographystyle{alpha}
\bibliography{common2}
 
\end{document}